\RequirePackage{fix-cm}
\documentclass[smallextended]{svjour3}       
\smartqed  
\usepackage{xspace}
\usepackage{amsfonts}
\usepackage{amsmath}
\usepackage{enumerate}
\usepackage{graphicx}
\usepackage{lscape}
\usepackage{bm}
\usepackage{color}
\usepackage{multicol}
\usepackage{cite}

\newcommand{\figref}[1]{{Figure~\ref{#1}}}

\newcommand{\secref}[1]{{Section~\ref{#1}}}
\renewcommand{\eqref}[1]{{(\ref{#1})}}

\newtheorem{assumption}[theorem]{Assumption}
\newcommand{\thmref}[1]{{Theorem~\ref{#1}}}
 \newcommand{\lemref}[1]{{Lemma~\ref{#1}}}
 \newcommand{\assref}[1]{{Assumption~\ref{#1}}}
\newcommand{\propref}[1]{{Proposition~\ref{#1}}}
\newcommand{\rmref}[1]{{Remark ~\ref{#1}}}
\usepackage{graphicx}
\begin{document}

\title{Convergence analysis of the Magnus-Rosenbrock type method for the finite element discretization of semilinear non-autonomous  parabolic PDEs with nonsmooth initial data}


\titlerunning{The Magnus-Rosenbrock type method for non-autonomous  parabolic PDE}        

\author{Antoine Tambue, Jean Daniel Mukam}
\authorrunning{A. Tambue, J. D. Mukam} 

\institute{A. Tambue (Corresponding author) \at
Department of Computing Mathematics and Physics, Western Norway University of Applied Sciences, Inndelsveien 28, 5063 Bergen, Norway\\
            Center for Research in Computational and Applied Mechanics (CERECAM), and Department of Mathematics and Applied Mathematics, University of Cape Town, 7701 Rondebosch, South Africa.\\
            The African Institute for Mathematical Sciences(AIMS) of South Africa and Stellenbosh University,\\
             Tel.: +27-785580321\\
   \email{antonio@aims.ac.za, antoine.tambue@vl.no} \\
   \and
   J. D. Mukam \at
           Fakult\"{a}t f\"{u}r Mathematik, Technische Universit\"{a}t Chemnitz, 09126 Chemnitz, Germany.  \\
            Tel.: +49-15213471370 \\
          \email{ jean.d.mukam@aims-senegal.org \\
          \hspace*{1cm}  jean-daniel.mukam@mathematik.tu-chemnitz.de}          
}

\date{Received: date / Accepted: date}

\maketitle

\begin{abstract}
This paper aims  to investigate a full numerical approximation of non-autonomous semilnear parabolic partial differential equations (PDEs) with nonsmooth initial data.
Our main interest is on such PDEs where the nonlinear part  is stronger than the linear part, also called reactive dominated transport equations.
For such equations, many classical  numerical methods lose their stability properties.
We perform  the space and time discretizations  respectively by the finite element method and an exponential integrator. 
We  obtain a novel explicit, stable  and  efficient scheme for such problems called Magnus-Rosenbrock 
method. We prove the convergence of the fully discrete scheme toward the exact solution. The result shows how 
the convergence orders in both space and time depend on the regularity of the initial data. In particular, when the initial data belongs 
to the domain of the family of the linear operator, we achieve convergence orders $\mathcal{O}\left(h^{2}+\Delta t^{2-\epsilon}\right)$, 
for an arbitrarily small $\epsilon>0$. Numerical simulations to illustrate our theoretical result are provided.
  
\end{abstract}

\section{Introduction}
\label{intro}
We consider the following abstract Cauchy problem with boundary conditions 
\begin{eqnarray}
\label{model}
u'(t)= A(t)u(t)+F(t,u(t)),  \quad u(0)=u_0, \quad t\in(0,T], \quad T>0,
\end{eqnarray}
on the Hilbert  space $H=L^2(\Lambda)$, where $\Lambda$ is an open bounded subset of $\mathbb{R}^d$ $(d=1,2,3)$.  The family of unbounded linear operators $A(t)$
is assumed to generate an analytic semigroup $S_s(t):=e^{A(s)t}$. Suitable assumptions on the nonlinear function $F$ and the linear operator $A(t)$
to ensure the existence of a unique mild solution of \eqref{model} are given in the following section. Equation of type \eqref{model} 
finds applications in many fields such as quantum fields theory, electromagnetism, nuclear physics, see e.g. \cite{Blanes}.
Since analytic solutions of \eqref{model} are usually not available, numerical algorithms are  the only tools to provide good approximations.
Numerical schemes for \eqref{model} with constant linear operator $A(t)=A$ are widely investigated in the scientific 
literature, see e.g. \cite{Stig1,Ostermann5,Hochbruck2,Ostermann3} and the references therein. If we turn our attention 
to the  non-autonomous case, the list of references becomes remarkably short. In the linear case, \eqref{model} has been 
investigated in \cite{Hochbruck1}, where the authors examined the convergence analysis of the Magnus integrator to Schr\"{o}dinger equation.
The Magnus integrator was further investigated in \cite{Ostermann1} for PDE \eqref{model} with $F$ independent of $u$,
where the authors applied the mid-point rule to approximate the Magnus expansion in order to achieve a second order approximation in time.
Numercal scheme for semilinear PDEs \eqref{model} was investigated in \cite{Ostermann4} 
and  the  convergence in time has been proved. In \cite{Ostermann4}, the authors used the backward Euler method.
Although  backward Euler method has good stability properties, it is computationally expensive as  nonlinear systems  need  to be solved at
each time step. Our goal here is to  provide a novel efficient scheme to solve \eqref{model} by upgrading
the  scheme  for linear PDEs in \cite{Ostermann1} and providing a mathematical rigorous convergence proof in space and in time. 
A standard direction to upgrade  the  Magnus integrator  \cite{Ostermann1} 
to semilinear PDEs consists to  keep the linear  structure  of \eqref{model} at each time step.
However, when the  linear part  of \eqref{model}  
is stronger than its  nonlinear part, the PDE \eqref{model} is driven by the linear part and the good 
stability properties of a scheme from such approach it is not guaranteed. 
Indeed when the   nonlinear part  of a PDE is stronger than its linear part,
the PDE is driven by its nonlinear part.
For such problems, keeping the linear  structure  of \eqref{model} at each step yields  schemes behaving  like the unstable explicit Euler 
method.

In this paper, we propose a novel numerical scheme  by applying the Rosenbrock-Type method 
\cite{Ostermann3,finance,Hochbruck2,Hairer,Antjd2} to the semi-discrete problem \eqref{semi1} 
combining with the Magnus-integrator to the linearized problem. This combination  yields an explicit efficient numerical method for such problems.
The linearization technique weakens the nonlinear part such that the linearized semi-discrete problem is driven by its new linear part.
In contrast to \cite{Ostermann4}, the linearization technique is done at every time step.
Note  that the Rosenbrock method was  investigated in the scientific literature only for autonomous problems, 
see e.g. \cite{Hochbruck2,Antjd2} for deterministic problem and recently  in \cite{Antjd1} 
for stochastic parabolic PDEs to the best of our knowledge. 
Moreover, the convergence analyses  in \cite{Ostermann1,Ostermann2,Ostermann4} are only in time.
 Furthermore,  we examine the space and time convergence with non smooth initial data where the space discretization 
is performed using the finite element method. 
Comparing with scheme in \cite{Antjd2}, the analysis here is extremely complicated due to the complexity of $A(t)$ and its semigroup $S_s(t)=e^{A(s)t}$.
This complexity is broken through novel rigorous mathematical results obtained in Section \ref{preliminaires}.
Furthermore, in contrast to the scheme  in \cite{Antjd2,Hochbruck},  the new scheme is second order accuracy in time  for non-autonomous PDEs \eqref{model} 
 with constant linear operator $A$  without the extra  matrix  exponential function  $\varphi_2$.
Our final convergence result shows how  the convergence orders in both space and time depend on the regularity of the initial data. In particular, when the initial data belongs 
to the domain of the family of the linear operator, we achieve convergence orders $\mathcal{O}\left(h^{2}+\Delta t^{2-\epsilon}\right)$, 
for an arbitrarily small $\epsilon>0$.
 
  The  paper is organized as follows. In  Section \ref{nummethod},  results about the  well posedness are provided along 
  with the Magnus-Rosenbrock scheme (MAGROS) and the main result. The proof of  the main result  
  is presented in Section \ref{proof1}. In Section \ref{numericalexperiment}, we present some numerical simulations to sustain our theoretical result.

\section{Mathematical setting and numerical method}
 \label{nummethod}
 \subsection{Notations, settings and well posedness}
Let us start by presenting briefly  notations, the main function
spaces and norms that  will be used in this paper. 
We denote by $\Vert \cdot \Vert$ the norm associated to
the inner product $\langle\cdot ,\cdot \rangle_H$ of the Hilbert space $H=L^{2}(\Lambda)$. The norm in the Sobolev space $H^m(\Lambda)$, $ m \geq 0$ will be denoted by
$\Vert. \Vert_m$. For a Hilbert space $U$ we denote by $\Vert\cdot\Vert_{U}$
the norm of $U$,
$L(U, H)$  the set of bounded linear operators  from
$U$ to $H$.  For ease of notation, we use $L(U,U)=:L(U)$.

To guarantee the existence of a unique mild solution of \eqref{model}, and for the purpose of the  convergence analysis, we make the following  assumptions.
\begin{assumption}
\label{assumption2}
\begin{itemize}
\item[(i)]
As in \cite{Ostermann1,Ostermann2,Gonz}, we assume that $\mathcal{D}\left(A(t)\right)=D$, $0\leq t\leq T$ and the family of linear operators $A(t) : D\subset H\longrightarrow H$ to be uniformly sectorial on $0\leq t\leq T$, i.e. there exist constants $c>0$ and $\theta\in\left(\frac{1}{2}\pi, \pi\right)$ such that
\begin{eqnarray}
\left\Vert \left(\lambda\mathbf{I}-A(t)\right)^{-1}\right\Vert_{L(L^2(\Lambda))}\leq \frac{c}{\vert \lambda\vert},\quad \lambda\in S_{\theta},
\end{eqnarray}
where $S_{\theta}:=\left\{\lambda\in\mathbb{C} : \lambda=\rho e^{i \phi}, \rho>0, 0\leq \vert \phi\vert\leq \theta\right\}$. As in \cite{Ostermann2}, by a standard scaling argument, we assume $-A(t)$ to be invertible with bounded inverse.
\item[(ii)]
 Similarly to \cite{Ostermann2,Ostermann1,Praha,Gonz}, we require the following Lipschitz conditions:  there exists a positive constant $K_1$ such that
 {\small
\begin{eqnarray}
\label{conditionB}
\left\Vert \left(A(t)-A(s)\right)(-A(0))^{-1}\right\Vert_{L(H)}&\leq& K_1\vert t-s\vert,\quad s,t\in[0, T],\\
\left\Vert  (-A(0))^{-1}\left(A(t)-A(s)\right)\right\Vert_{L(D,H)}&\leq& K_1\vert t-s\vert,\quad s,t\in[0, T].
\end{eqnarray}
}
\item[(iii)]
 Since we are dealing with non smooth data, we follow \cite{Praha} and  assume that 
\begin{eqnarray}
\label{domaine}
\mathcal{D}\left(\left(-A(t)\right)^{\alpha}\right)=\mathcal{D}\left(\left(-A(0)\right)^{\alpha}\right),\quad 0\leq t\leq T,\quad 0\leq \alpha\leq 1
\end{eqnarray}
and there exists a positive constant $K_2$ such that  the following estimate holds uniformly for $t\in[0,T]$
{\small
\begin{eqnarray}
\label{equivnorme1}
K_2^{-1}\left\Vert \left(-A(0)\right)^{\alpha}u\right\Vert\leq \left\Vert (-A(t))^{\alpha}u\right\Vert\leq K_2\left\Vert (-A(0))^{\alpha}u\right\Vert,\quad u\in \mathcal{D}\left(\left(-A(0)\right)^{\alpha}\right).
\end{eqnarray} 
}
\item[(iv)]
Similarly to \cite[(3.17)]{Ostermann2} and \cite{Gonz,Ostermann4}, we assume that the map $t \longmapsto A(t)$ is twice differentiable  and for any $\alpha_1,\alpha_2\in[0,1]$ such that $\alpha_1+\alpha_2=1$,   the following estimates are satisfied
{\small
\begin{eqnarray*}
\Vert (-A(s))^{-\alpha_1}A''(t)A(s)^{-\alpha_2}\Vert_{L\left(\left(-A(0)\right)^{1-\alpha_2}, H\right)}&\leq& C_0,\quad s,t\in[0,T],\\
\Vert (-A(0))^{-\alpha_1}(A(t)-A(s))(-A(0))^{-\alpha_2}\Vert_{L\left(\left(-A(0)\right)^{1-\alpha_2}, H\right)}&\leq& C_0\vert t-s\vert,\quad s,t\in[0,T],
\end{eqnarray*}
}
where $C_0$ is a positive constant independent of $t_1$ and $t_2$. 
\end{itemize}
\end{assumption}
\begin{remark}
\label{remark1}
From \assref{assumption2} (i) and (iii), it follows that for all $\alpha\geq 0$ and $\delta\in[0,1]$, 
there exists a constant $C_1>0$ such that the following estimates hold uniformly for  $t\in[0,T]$
\begin{eqnarray}
\label{smooth}
\left\Vert (-A(t))^{\alpha}e^{sA(t)}\right\Vert_{L(H)}\leq C_1s^{-\alpha},\quad  \quad s>0,\\
\label{smootha}
 \left\Vert(-A(t))^{-\delta}\left(\mathbf{I}-e^{sA(t)}\right)\right\Vert_{L(H)}\leq C_1s^{\delta}, \quad\quad\quad s\geq 0,
\end{eqnarray}
see e.g. \cite[(2.1)]{Ostermann2}.
\end{remark}
\begin{remark}
\label{remark2}
Let $\Delta(T):=\{(t,s) : 0\leq s\leq t\leq T\}$.
It is well known that \cite[Theorem 6.1, Chapter 5]{Pazy}  under \assref{assumption2} there exists a unique evolution system \cite[Definition 5.3, Chapter 5]{Pazy} $U : \Delta(T)\longrightarrow L(H)$ such that
\begin{itemize}
\item[(i)] There exists a positive constant $K_0$ such that
\begin{eqnarray}
\Vert U(t,s)\Vert_{L(H)}\leq K_0,\quad 0\leq s\leq t\leq T.
\end{eqnarray}
\item[(ii)] $U(.,s)\in C^1(]s,T] ; L(H))$, $0\leq s\leq T$,
\begin{eqnarray}
\frac{\partial U}{\partial t}(t,s)=-A(t)U(t,s), \quad 0\leq s\leq t\leq T, \\
\Vert A(t)U(t,s)\Vert_{L(H)}\leq \frac{K_0}{t-s},\quad 0\leq s<t\leq T.
\end{eqnarray}
\item[(iii)] $U(t,.)v\in C^1([0,t[ ; H)$, $0<t\leq T$, $v\in\mathcal{D}(A(0))$ and 
\begin{eqnarray}
\frac{\partial U}{\partial s}(t,s)v=-U(t,s)A(s)v,\quad 0\leq s\leq t\leq T,\\
 \Vert A(t)U(t,s)A(s)^{-1}\Vert_{L(H)}\leq K_0, \quad 0\leq s\leq t\leq T.
\end{eqnarray}
\end{itemize}
\end{remark}
 We equip $V_{\alpha}(t) : = \mathcal{D}\left(\left(-A(t)\right)^{\alpha/2}\right)$, $\alpha\in \mathbb{R}$ with the norm $\Vert u\Vert_{\alpha,t} := \Vert (-A(t))^{\alpha/2}u\Vert$. Due to \eqref{domaine}-\eqref{equivnorme1} and for the seek of  ease notations, we simply write $V_{\alpha}$ and $\Vert .\Vert_{\alpha}$ instead of $V_{\alpha}(t)$ and $\Vert .\Vert_{\alpha,t}$ respectively.
 
 \begin{assumption}
 \label{assumption1}
The initial data $u_0 : \Lambda\longrightarrow H$ is assumed to  satisfy $u_0\in \mathcal{D}\left(\left(-A(0)\right)^{\beta/2}\right)$, $0\leq \beta\leq 2$.
 \end{assumption}
 
 Similarly to \cite[(8.1.1)]{Lunardi}, \cite{Ostermann4} and \cite[(5.3)]{Stig2}, we make the following assumption on the nonlinear function.
\begin{assumption}
\label{assumption3}
The function $F : [0,T]\times H\longrightarrow H$ is assumed to be twice differentiable with respect to the first and second variables and with bounded partial derivatives, i.e. there exists $K_3\geq 0$ such that for $k=\{1,2\}$ we have
{\small
\begin{eqnarray}
\left\Vert\frac{\partial^2F}{\partial t\partial u}(t,u)\right\Vert_{L(H)}&\leq& K_3,\quad\left\Vert\frac{\partial^kF}{\partial t^k} (t,u)\right\Vert\leq K_3(1+\Vert u\Vert),\quad t\in[0,T],\quad u\in H,\\
 \left\Vert\frac{\partial F}{\partial u} (t,u)\right\Vert_{L(H)}&\leq& K_3,\quad \left\Vert\frac{\partial^2 F}{\partial u^2} (t,u)\right\Vert_{L(H\times H, H)}\leq K_3,\quad t\in[0,T],\quad u\in H.
\end{eqnarray} 
}

Moreover, we assume assume $F'(t, u)$ to be coercive for $t\in[0, T]$ and $u\in H$, i.e.  there exists $\kappa > -b_0$ such that
\begin{eqnarray}
\label{prix}
-\left\langle F'(t,u)v, v\right\rangle_H\geq \kappa \Vert v\Vert^2,\quad t\in[0, T],\quad v, u\in H,\\
b_0= \underset{ t \geq 0} {\inf}  \{ \text{Re} (\lambda (t)),\, \lambda (t) \in \sigma (A(t)) \,\,(\text{spectrum  of }  A(t))\}\,
\end{eqnarray}
where $F'(t,u):=\frac{\partial F}{\partial u}(t,u)$.  We also assume   the nonlinear function  $F$ to satisfy the  Lipschitz condition, i.e. there exists a constant $K_4\geq 0$ such that
\begin{eqnarray}
\label{Lipschitz}
\Vert F(t,u)-F(s,v)\Vert\leq K_4(\vert t-s\vert+\Vert u-v\Vert),\quad s,t\in[0,T],\quad u,v\in H.
\end{eqnarray}
\end{assumption}
Indeed from the coercivity \eqref{ellip2}, we can take  $b_0= \lambda_0$.

The following theorem provides the well posedness of problem \eqref{model}.
\begin{theorem}
\label{theorem1}
Let \assref{assumption1}, \assref{assumption2} and \assref{assumption3} be fulfilled. Then the initial value problem \eqref{model} has a unique mild solution $u(t)$ given by
\begin{eqnarray}
\label{mild0}
u(t)=U(t,0)u_0+\int_0^tU(t,s)F(s,u(s))ds,\quad t\in(0,T],
\end{eqnarray}
where $U(t,s)$ is the evolution system defined in \rmref{remark2}. Moreover, the following space regularity holds
\begin{eqnarray}
\label{spacereg1}
\Vert (-A(0))^{\beta/2}u(t)\Vert\leq C\left(1+\Vert (-A(0))^{\beta/2}u_0\Vert\right),\quad \beta\in[0,2),\quad t\in[0,T].
\end{eqnarray}
\begin{proof}
\thmref{theorem1} is an extension of \cite[Chapter 5, Theorem 7.1]{Pazy} to the full semilinear problem.
Its  proof can be done  using arguments based  on a fixed point theorem and the Gronwall's lemma as  of \cite[Chpater 6, Theorem 1.2]{Pazy}.
The proof of \eqref{spacereg1} follows from  the regularities estimates of the evolution parameter $U(t,s)$.
\end{proof}
\end{theorem}

\subsection{Finite element discretization}
\label{numscheme}
For the seek of simplicity, we assume the family of linear operators $A(t)$ to be of second order and has the following form
\begin{eqnarray}
\label{family}
A(t)u=\sum_{i,j=1}^d\frac{\partial}{\partial x_i}\left(q_{ij}(x,t)\frac{\partial u}{\partial x_j}\right)-\sum_{j=1}^dq_j(x,t)\frac{\partial u}{\partial x_j}.
\end{eqnarray}
We require the coefficients $q_{i,j}$ and $q_j$ to be smooth functions of the variable $x\in\overline{\Lambda}$ and H\"{o}lder-continuous with respect to $t\in[0,T]$. We further assume that there exists a positive constant $c$ such that the following  ellipticity condition holds
\begin{eqnarray}
\label{ellip}
\sum_{i,j=1}^dq_{ij}(x,t)\xi_i\xi_j\geq c\vert \xi\vert^2, \quad (x,t)\in\overline{\Lambda}\times [0,T].
\end{eqnarray}
 Under the above assumptions on $q_{ij}$ and $q_j$, it is  well known  that  the family of linear operators defined by \eqref{family} fulfills  \assref{assumption2} (i)-(ii)  with $D=H^2(\Lambda)\cap H^1_0(\Lambda)$, see \cite[Section 7.6]{Pazy} or \cite[Section 5.2]{Tanabe}. The above assumptions on $q_{ij}$ and $q_j$ also imply that \assref{assumption2} (iii) is fulfilled, see e.g. \cite[Example 6.1]{Praha} or \cite{Amann,Seely}. 
 
 As in \cite{Suzuki,Antonio1}, we introduce two spaces $\mathbb{H}$ and $V$, such that $\mathbb{H}\subset V$,  depending on the boundary conditions for the domain of the operator $-A(t)$ and the corresponding bilinear form. For  Dirichlet  boundary conditions we take 
\begin{eqnarray}
V=\mathbb{H}=H^1_0(\Lambda)=\{v\in H^1(\Lambda) : v=0\quad \text{on}\quad \partial \Lambda\}.
\end{eqnarray}
For Robin  boundary condition and  Neumann  boundary condition, which is a special case of Robin boundary condition ($\alpha_0=0$), we take $V=H^1(\Lambda)$ and
\begin{eqnarray}
\mathbb{H}=\{v\in H^2(\Lambda) : \partial v/\partial v_A+\alpha_0v=0,\quad \text{on}\quad \partial \Lambda\}, \quad \alpha_0\in\mathbb{R}.
\end{eqnarray}
Using  Green's formula and the boundary conditions, we obtain the corresponding bilinear form associated to $-A(t)$  
\begin{eqnarray*}
a(t)(u,v)=\int_{\Lambda}\left(\sum_{i,j=1}^dq_{ij}(x,t)\dfrac{\partial u}{\partial x_i}\dfrac{\partial v}{\partial x_j}+\sum_{i=1}^dq_i(x,t)\dfrac{\partial u}{\partial x_i}v\right)dx, \quad u,v\in V,
\end{eqnarray*}
for Dirichlet boundary conditions and  
\begin{eqnarray*}
a(t)(u,v)=\int_{\Lambda}\left(\sum_{i,j=1}^dq_{ij}(x,t)\dfrac{\partial u}{\partial x_i}\dfrac{\partial v}{\partial x_j}+\sum_{i=1}^dq_i(x,t)\dfrac{\partial u}{\partial x_i}v\right)dx+\int_{\partial\Lambda}\alpha_0uvdx.
\end{eqnarray*}
for Robin  and Neumann boundary conditions. 
Using  G\aa rding's inequality, it holds that there exist two constants $\lambda_0$ and $c_0$ such that
\begin{eqnarray}
a(t)(v,v)\geq \lambda_0\Vert v \Vert^2_{1}-c_0\Vert v\Vert^2, \quad  v\in V,\quad t\in[0,T].
\end{eqnarray}
By adding and subtracting $c_{0}u $ on the right hand side of (\ref{model}), we obtain a new family of linear operators that we still denote by  $A(t)$. Therefore the  new corresponding   bilinear form associated to $-A(t)$ still denoted by $a(t)$ satisfies the following coercivity property
\begin{eqnarray}
\label{ellip2}
a(t)(v,v)\geq \; \lambda_0\Vert v\Vert_{1}^{2},\;\;\;\;\; v \in V,\quad t\in[0,T].
\end{eqnarray}
Note that the expression of the nonlinear term $F$ has changed as we included the term $-c_{0}u$
in a new nonlinear term that we still denote by $F$.

The coercivity property (\ref{ellip2}) implies that $A(t)$ is sectorial on $L^2(\Lambda)$, see e.g. \cite{Stig2}. Therefore   $A(t)$ generates an analytic semigroup   $S_t(s)=e^{s A(t)}$  on $L^{2}(\Lambda)$  such that \cite{Henry}
\begin{eqnarray}
S_t(s)= e^{s A(t)}=\dfrac{1}{2 \pi i}\int_{\mathcal{C}} e^{ s\lambda}(\lambda I - A(t))^{-1}d \lambda,\;\;\;\;\;\;\;
\;s>0,
\end{eqnarray}
where $\mathcal{C}$  denotes a path that surrounds the spectrum of $A(t)$.
The coercivity  property \eqref{ellip2} also implies that $-A(t)$ is a positive operator and its fractional powers are well defined and 
  for any $\alpha>0$ we have
\begin{equation}
\label{fractional}
 \left\{\begin{array}{rcl}
         (-A(t))^{-\alpha} & =& \frac{1}{\Gamma(\alpha)}\displaystyle\int_0^\infty  s^{\alpha-1}{\rm e}^{sA(t)}ds,\\
         (-A(t))^{\alpha} & = & ((-A(t))^{-\alpha})^{-1},
        \end{array}\right.
\end{equation}
where $\Gamma(\alpha)$ is the Gamma function (see \cite{Henry}).  
 The domain  of $(-A(t))^{\alpha/2}$  are  characterized in \cite{Suzuki,Stig1,Stig2}  for $1\leq \alpha\leq 2$  with equivalence of norms as follows.
\begin{eqnarray}
\mathcal{D}((-A(t))^{\alpha/2})&=&H^1_0(\Lambda)\cap H^{\alpha}(\Lambda)\hspace{1cm} 
\text{(for Dirichlet boundary condition)}\nonumber\\
\mathcal{D}(-A(t))&=&\mathbb{H},\quad \mathcal{D}((-A(t))^{1/2})=H^1(\Lambda)\hspace{0.5cm} \text{(for Robin boundary condition)}\nonumber\\
\Vert v\Vert_{H^{\alpha}(\Lambda)}&\equiv& \Vert ((-A(t))^{\alpha/2}v\Vert:=\Vert v\Vert_{\alpha},\quad  v\in \mathcal{D}((-A(t))^{\alpha/2}).\nonumber
\end{eqnarray}
The characterization of $\mathcal{D}((-A(t))^{\alpha/2})$ for $0\leq \alpha<1$ can be found in  \cite[Theorem 2.1 \& Theorem 2.2]{Nambu}.

Let us now move to the space  approximation of  problem \eqref{model}. We start with the discretization  of our domain $\Lambda$ by a finite triangulation.
Let $\mathcal{T}_h$ be a triangulation with maximal length $h$. Let $V_h \subset V$ denotes the space of continuous and piecewise 
linear functions over the triangulation $\mathcal{T}_h$. As in \cite[(1.6)]{Luskin}, we assume that
\begin{eqnarray}
\label{ritz0}
\inf_{\phi_h\in V_h}\Vert v-\phi_h\Vert_j\leq Ch^{r-j}\Vert v\Vert_r,\quad v\in V\cap H^r(\Lambda),\quad r\in\{1,2\},
\end{eqnarray}
for all $j\in\{0,1\}$. Moreover, we assume that
\begin{eqnarray}
\label{ritz0a}
\inf_{\phi_h\in V_h}\Vert v-\phi_h\Vert_2\leq C\Vert v\Vert_2,\quad v\in V\cap H^2(\Lambda).
\end{eqnarray}
 We consider the projection $P_h$  defined  from  $H=L^2(\Lambda)$ to $V_h$ by 
\begin{eqnarray}
\label{discrete1}
(P_hu,\chi)=(u,\chi), \quad  \chi\in V_h,\,  u\in H.
\end{eqnarray}
For all $t\in[0, T]$, the discrete operator $A_h(t) : V_h\longrightarrow V_h$ is defined by 
\begin{eqnarray}
\label{discrete2}
(A_h(t)\phi,\chi)=(A(t)\phi,\chi)=-a(t)(\phi,\chi),\quad  \phi,\chi\in V_h.
\end{eqnarray}
The coercivity property \eqref{ellip2}  implies that there exist two constants $C_2>0$ and $\theta\in(\frac{1}{2}\pi,\pi)$ such that (see e.g. \cite[(2.9)]{Stig2} or \cite{Suzuki,Henry})
 \begin{eqnarray}
 \label{sectorial1}
 \Vert (\lambda\mathbf{I}-A_h(t))^{-1}\Vert_{L(H)}\leq \frac{C_2}{\vert \lambda\vert},\quad \lambda \in S_{\theta}
 \end{eqnarray}
 holds uniformly for $h>0$ and $t\in[0,T]$. The coercivity condition \eqref{ellip2} implies that for any $t\in[0,T]$, $A_h(t)$ generates an analytic semigroup $S^h_t(s):=e^{sA_h(t)}$, $s\in[0,T]$. The coercivity property \eqref{ellip2} also implies that the smooth properties \eqref{smooth}  and \eqref{smootha} hold for $A_h$ uniformly for $h>0$ and $t\in[0,T]$, i.e. for all $\alpha\geq 0$ and $\delta\in[0,1]$, there exists a positive constant $C_3$ such that the following estimates hold uniformly for $h>0$ and $t\in[0,T]$, see e.g. \cite{Suzuki,Henry}
 \begin{eqnarray}
 \label{smooth2}
 \left\Vert(-A_h(t))^{\alpha}e^{sA_h(t)}\right\Vert_{L(H)}&\leq& C_3s^{-\alpha}, \quad s>0, \\
 \label{smooth1}
  \left\Vert (-A_h(t))^{-\delta}\left(\mathbf{I}-e^{sA_h(t)}\right)\right\Vert_{L(H)}&\leq& C_3s^{\delta}, \quad s\geq 0.
 \end{eqnarray}
The semi-discrete in space of problem \eqref{model} consists of  finding $u^h(t)\in V_h$ such that 
\begin{eqnarray}
\label{semi1}
\dfrac{du^h(t)}{dt}=A_h(t)u^h(t)+P_hF(t,u^h(t)), \quad u^h(0)=P_hu_0,\quad t\in(0,T].
\end{eqnarray}
\subsection{Fully discrete scheme and main result}
Throughout this paper, without loss of generality, we use  a  fixed time step  $\Delta t=T/M$, $M\in\mathbb{N}$   and we set  $t_m=m\Delta t$, $m\in \mathbb{N}$. 
The time discretization consists  of  computing the  numerical approximation $u^h_m$ of $u^h(t_m)$ at discrete times $t_m=m\Delta t \in (0,T]$,  $\Delta t >0$, $m=0,\cdots,M$. 
Let us build an explicit scheme, efficient to solve \eqref{model}. 
The method is based on the following linearisation of  \eqref{semi1} at each time step, aiming to weaken the nonlinear part
\begin{eqnarray}
\label{semi}
\dfrac{du^h(t)}{dt}=\left[A_h(t)+J_m^h\right]u^h(t)+a^h_mt+G^h_m(t,u^h(t)), \quad t_m\leq t\leq t_{m+1},
\end{eqnarray}
for $m=0,\cdots, M-1$, 
where the derivatives $J_m^h$ and $a^h_m$  are  respectively the partial derivatives of $F$ at $\left(t_m+\frac{\Delta t}{2}, u^h_m\right)$ with respect to $u$ and $t$, given by
\begin{eqnarray}
\label{remainder1}
J^h_m :=P_h\frac{\partial F}{\partial u}\left(t_m+\frac{\Delta t}{2},u^h_m\right)\quad \text{and}\quad a^h_m:=P_h\frac{\partial F}{\partial t}\left(t_m+\frac{\Delta t}{2}, u^h_m\right)
\end{eqnarray}
and the remainder  $G^h_m$ is given by 
\begin{eqnarray}
\label{remainder2}
 G^h_m(t,u^h(t)) :=P_hF(t,u^h(t))-J_m^hu^h(t)-a^h_mt.
\end{eqnarray}
Note that using \assref{assumption3} the following estimate holds
\begin{eqnarray}
\label{remainder3}
\Vert J^h_mu-J^h_mv\Vert_{L(H)}\leq K_3\Vert u-v\Vert,\quad u, v\in H, \quad h>0,\quad m=0,\cdots,M.
\end{eqnarray}
It follows therefore from \eqref{remainder3}, \eqref{Lipschitz} and \eqref{remainder2} that the remainder $G^h_m$ satisfies the following Lipschitz estimate
\begin{eqnarray}
\label{remainder4}
\Vert G^h_m(t,u)-G^h_m(t,v)\Vert \leq (K_3+K_4)\Vert u-v\Vert,\quad u,v \in H,\quad t\in[0,T].
\end{eqnarray}
 Applying the  exponential-like Euler and Midpoint integrators \cite{LEM} to \eqref{semi} gives the following numerical scheme, called Magnus-Rosenbrock method (MAGROS)
 {\small
\begin{eqnarray}
\label{erem}
u^h_{m+1}&=& e^{\Delta t\left(A_{h,m}+J^h_m\right)}u^h_m+\Delta t\varphi_1\left(\Delta t(A_{h,m}+J^h_m)\right)a^h_m\left(t_m+\frac{\Delta t}{2}\right)\nonumber\\
&+&\Delta t\varphi_1\left(\Delta t(A_{h,m}+J^h_m)\right)G^h_m\left(t_m+\frac{\Delta t}{2},u^h_m\right),\quad m=0,\cdots,M-1,
\end{eqnarray}
}
where the linear  operator $A_{h,m}$ is given by
\begin{eqnarray}
\label{defA}
A_{h,m}:=A_h\left(t_m+\frac{\Delta t}{2}\right)
\end{eqnarray} 
and the linear function $\varphi_1$ is given by
\begin{eqnarray}
\label{phi1}
\quad\varphi_1\left(\Delta t\left(A_{h,m}+J^h_m\right)\right)&:=&\frac{1}{\Delta t}\int_0^{\Delta t}e^{\left(A_{h,m}+J^h_m\right)(\Delta t-s)}ds. 
\end{eqnarray}
Note that the numerical scheme \eqref{erem} can be written in the following form, efficient for simulation
\begin{eqnarray}
\label{erem1}
u^h_{m+1}=u^h_m+\Delta t\varphi_1\left(\Delta t(A_{h,m}+J^h_m)\right)\left[A_{h,m}u^h_m+P_hF\left(t_m+\frac{\Delta t}{2},u^h_m\right)\right].
\end{eqnarray}
The numerical scheme \eqref{erem} can also be written in the following integral form, useful for the error analysis
\begin{eqnarray}
\label{semi2}
u^h_{m+1}&=&e^{\Delta t\left(A_{h,m}+J^h_m\right)}u^h_m+\int_0^{\Delta t}e^{\left(A_{h,m}+J^h_m\right)(\Delta t-s)}a^h_m\left(t_m+\frac{\Delta t}{2}\right)ds\nonumber\\
&+&\int_0^{\Delta t}e^{\left(A_{h,m}+J^h_m\right)(\Delta t-s)}G^h_m\left(t_m+\frac{\Delta t}{2},u^h_m\right)ds.
\end{eqnarray}
 We will need the following further assumption on the nonlinearity, useful to achieve full convergence order $2$  
 in space without any logarithmic perturbation  when $u_0\in \mathcal{D}(-A(0))$. This assumption was also used in \cite[Remark 2.9]{Antonio1}.
\begin{assumption}
\label{assumption4}
We assume that $ F:[0,T]\times H\longrightarrow H$ satisfies the following estimate
\begin{eqnarray}
\Vert (-A(s))^{\gamma}F(t, u(r))\Vert \leq C(\gamma)\left(1+\Vert (-A(s))^{\gamma}u(r)\Vert\right),\quad s,r,t\in[0,T],
\end{eqnarray}
for any $\gamma>0$ small enough.
\end{assumption}

We can now state our convergence result, which is in fact the main result of this paper.
\begin{theorem}\textbf{[Main result]}
\label{mainresult1}
Let \assref{assumption2}, \assref{assumption1} and  \assref{assumption3}  be fulfilled.
\begin{itemize}
\item[(i)] If $0<\beta<2$, then the following error estimate holds
\begin{eqnarray}
\Vert u(t_m)-u^h_m\Vert\leq C\left(h^{\beta}+\Delta t^{1+\beta/2-\epsilon}\right),
\end{eqnarray}
where $\epsilon>0$ is a positive constant small enough.  
\item[(ii)]
If $\beta=2$, then the following error estimate holds
\begin{eqnarray}
\Vert u(t_m)-u^h_m\Vert\leq C\left(h^{2}\left(1+\max\left(0, \ln(t_m/h)\right)\right)+\Delta t^{2-\epsilon}\right).
\end{eqnarray}
\item[(iii)]
If $\beta=2$ and moreover if \assref{assumption4} is fulfilled then the following error estimate holds
\begin{eqnarray}
\Vert u(t_m)-u^h_m\Vert\leq C\left(h^{2}+\Delta t^{2-\epsilon}\right).
\end{eqnarray}
\end{itemize}
\end{theorem}
\begin{remark}
\thmref{mainresult1} extends the result in \cite{Ostermann1} to a fully semilinear problem with nonsmooth initial data. 
Note that the linearisation technique  allows to achieve convergence order almost $2$ when $u_0\in\mathcal{D}(-A(0))$. 
\end{remark}

\section{Proof of the main result}
\label{proof1}
\subsection{Preliminaries results}
\label{preliminaires}

The following lemma will be useful in our convergence proof.
\begin{lemma}
\label{lemma0}
Let \assref{assumption2} be fulfilled. Then for any $\gamma\in[0,1]$ the following estimates hold
{\small
\begin{eqnarray}
\label{equidiscrete1}
K^{-1}\Vert (-(A_h(0))^{-\gamma}v\Vert&\leq& \Vert ((-A_h(t))^{-\gamma}v\Vert\leq K\Vert ((-A_h(0))^{-\gamma}v\Vert,\quad v\in V_h,\\
\label{equidiscrete2}
K^{-1}\Vert (-(A_h(0))^{\gamma}v\Vert&\leq& \Vert ((-A_h(t))^{\gamma}v\Vert\leq K\Vert ((A_h(0))^{\gamma}v\Vert,\quad v\in V_h,
\end{eqnarray}
}
 uniformly in $h>0$ and $t\in[0,T]$, where $K$ is a positive constant independent of $t$ and $h$.
\end{lemma}

\begin{proof}
We only prove \eqref{equidiscrete1} since the proof of \eqref{equidiscrete2} is similar to \cite[Lemma 1]{Antjd1} by using \assref{assumption2} (iii). 
 For relatively smooth coefficients ($q_j\in C^1(\Lambda)$), the formal adjoint of $A(t)$ denoted by $A^*(t)$ is given by (see e.g. \cite[Section 6.2.3]{Evans})
 {\small
 \begin{eqnarray}
 \label{adjoint1}
 A^*(t)=\sum_{i,j=1}^d\frac{\partial}{\partial x_j}\left(q_{ij}(x,t)\frac{\partial}{\partial x_i}\right)+\sum_{j=1}^dq_j(x,t)\frac{\partial}{\partial x_j}+\left(\sum_{j=1}^d\frac{\partial q_j}{\partial x_j}(x,t)\right)\mathbf{I},
 \end{eqnarray}
 }
 for any $t\in[0,T]$.
 It follows therefore from \eqref{adjoint1} that $\mathcal{D}(-A^*(t))=\mathcal{D}(-A(t))$ for all $t\in[0,T]$.
It also follows from \eqref{adjoint1}  that  the coefficients of $A^*(t)$ satisfy the same assumptions as that of $A(t)$.  Therefore  from \cite[Example 6.1]{Praha} or \cite{Amann,Seely} it holds that $A^*(t)$ satisfies  \assref{assumption2} (iii). More precisely, for all $\alpha\in[0,1]$ and $t\in[0,1]$,  $\mathcal{D}((-A^*(t))^{\alpha})=\mathcal{D}((-A^*(0))^{\alpha})$  and for all $v\in\mathcal{D}((-A^*(0))^{\alpha})$ it holds that
\begin{eqnarray}
 \label{adjoint2}
 C^{-1}\Vert (-A^*(0))^{\alpha}v\Vert\leq \Vert (-A^*(t))^{\alpha}v\Vert\leq C\Vert (-A^*(0))^{\alpha}v\Vert,\quad t\in[0,T].
 \end{eqnarray} 
 Note that for all $t\in[0,T]$,  $(A^*(t))_h=A_h^*(t)$, where $(A^*(t))_h$ stands for the discrete operator associated to $A^*(t)$ and $A_h^*(t)$ is the adjoint of $A_h(t)$. Indeed using \eqref{discrete2}, it holds that
 \begin{eqnarray}
 \label{adjoint2b}
 \langle (A^*(t))_hv,\chi\rangle_H&=&\langle A^*(t)v,\chi\rangle_H=\langle v, A(t)\chi\rangle_H=\langle A(t)\chi, v\rangle_H\nonumber\\
 &=&\langle A_h(t)\chi,v\rangle_H=\langle \chi, A_h^*(t)v\rangle_H\nonumber\\
 &=&\langle A_h^*(t)v,\chi\rangle_H,\quad \chi, v\in V_h, \quad t\in[0,T],
 \end{eqnarray}
 and therefore  $(A^*(t))_h=A_h^*(t)$ for all $t\in[0,T]$.  Let us recall the following  equivalence of norms \cite[(2.12)]{Stig2}, where we replace  $A$ by $A^*(t)$
\begin{eqnarray}
\label{adjoint2ad}
\Vert (-A_h^*(t))^{1/2}v\Vert\approx \Vert(- A^*(t))^{1/2}v\Vert,\quad v\in V_h, \quad t\in[0,T].
\end{eqnarray}
Using \eqref{adjoint2} and \eqref{adjoint2ad}  it holds that there exists a positive constant $K$ such that 
\begin{eqnarray}
\label{equisiam1}
K^{-1}\Vert (-(A_h^*(0))^{1/2}v\Vert\leq \Vert ((-A_h^*(t))^{1/2}v\Vert\leq K\Vert ((-A_h^*(0))^{1/2}v\Vert,
\end{eqnarray}
for any $t\in[0,T]$ and $v\in V_h$.
Following closely \cite{Stig2} or \cite[(3.7)]{Stig3}, it holds that
\begin{eqnarray}
\label{adjoint3} 
\Vert(- A_h(t))^{-1/2}v\Vert&=&\sup_{v_h\in V_h}\frac{\vert\langle (-A_h(t))^{-1/2}v,v_h\rangle_H\vert}{\Vert v_h\Vert}\nonumber\\
&=&\sup_{v_h\in V_h}\frac{\vert \langle v, (-A_h^*(t))^{-1/2}v_h\rangle_H\vert}{\Vert v_h\Vert}\nonumber\\
&=&\sup_{w_h\in V_h}\frac{\vert \langle v, w_h\rangle_H\vert}{\Vert (-A_h^*(t))^{1/2}w_h\Vert},\quad v\in V_h.
\end{eqnarray}
Using \eqref{equisiam1} yields
\begin{eqnarray}
\label{eclair1}
\sup_{w_h\in V_h}\frac{\vert \langle v, w_h\rangle_H\vert}{K\Vert  (-A_h^*(0))^{1/2}w_h\Vert}&\leq& \sup_{w_h\in V_h}\frac{\vert \langle v, w_h\rangle_H\vert}{\Vert (-A_h^*(t))^{1/2}w_h\Vert}\nonumber\\
&\leq& K\sup_{w_h\in V_h}\frac{\vert \langle v, w_h\rangle_H\vert}{\Vert (-A_h^*(0))^{1/2}w_h\Vert}
\end{eqnarray}
Combining \eqref{adjoint3}  with \eqref{eclair1} yields
{\small
\begin{eqnarray}
\label{equidiscrete3}
K^{-1}\Vert (-A_h(0))^{-1/2}v\Vert\leq \Vert (-A_h(t))^{-1/2}v\Vert\leq\Vert (-A_h(0))^{-1/2}v\Vert,\quad v\in V_h
\end{eqnarray}
}
for all $t\in[0,T]$.
Note that \eqref{equidiscrete3} obviously holds if we replace $1/2$ by $0$ and by $1$. The proof of the lemma is therefore completed by interpolation theory.
\end{proof}
For  $t\in[0,T]$,  we  introduce the Ritz projection $R_h(t) :V\longrightarrow V_h$ defined by 
\begin{eqnarray}
\label{ritz1}
\langle -A(t)R_h(t)v,\chi\rangle_H=\langle -A(t)v,\chi\rangle_H=a(t)(v,\chi),\quad v\in V,\quad \chi\in V_h.
\end{eqnarray}
Under the regularity assumptions on the triangulation \eqref{ritz0} and in view of the V-ellipticity condition \eqref{ellip}, 
it is well known (see e.g.  \cite[(3.2)]{Luskin} or \cite{Ciarlet,Suzuki}) that the following error estimate holds
\begin{eqnarray}
\label{ritz2}
\Vert R_h(t)v-v\Vert+h\Vert R_h(t)v-v\Vert_{H^1(\Lambda)}\leq Ch^{r}\Vert v\Vert_{H^{r}(\Lambda)},\quad v\in V\cap H^{r}(\Lambda),
\end{eqnarray}
for any $r\in[1,2]$. Moreover, using \eqref{ritz0a} it holds that
\begin{eqnarray}
\label{ritz0b}
\Vert R_h(t)v-v\Vert_{H^2(\Lambda)}\leq C\Vert v\Vert_2,\quad v\in V\cap H^2(\Lambda),\quad t\in[0,T].
\end{eqnarray}
The following error estimate also holds (see e.g. \cite[(3.3)]{Luskin} or \cite{Ciarlet,Suzuki}) 
{\small
\begin{eqnarray}
\label{ritz3}
\Vert D_t\left(R_h(t)v-v\right)\Vert+h\Vert D_t\left(R_h(t)v-v\right)\Vert_{H^1(\Lambda)}\leq Ch^{r}\left(\Vert v\Vert_{H^{r}(\Lambda)}+\Vert D_tv\Vert_{H^{r}(\Lambda)}\right),
\end{eqnarray}
}
for any $r\in[1,2]$ and $v\in V\cap H^r(\Lambda)$, where $D_t:=\frac{\partial }{\partial t}$. The following lemma will be useful in our convergence proof.
\begin{lemma}
\label{lemma0a} 
Under \assref{assumption2}, the following estimates hold
{\small
\begin{eqnarray}
\label{ref3}
\Vert (A_h(t)-A_h(s))(-A_h(r))^{-1}u^h\Vert&\leq& C\vert t-s\vert\Vert u^h\Vert,\quad r,s,t\in[0,T],\quad u^h\in V_h,\\
\label{ref2}
\Vert (-A_h(r))^{-1}\left(A_h(s)-A_h(t)\right)u^h\Vert&\leq& C\vert s-t\vert\Vert u^h\Vert,\quad r,s,t\in[0,T],\quad u^h\in V_h\cap D.
\end{eqnarray}
}
Moreover for any $u^h\in V_h\cap \mathcal{D}\left(\left(-A(0)\right)^{1-\alpha_2}\right)$ the following estimate holds
{\small
\begin{eqnarray}
\label{ref4}
\Vert (-A_h(0))^{-\alpha_1}(A_h(t)-A_h(s))(-A_h(0))^{-\alpha_2}u^h\Vert&\leq& C\vert t-s\vert\Vert u^h\Vert,\quad s,t\in[0,T].
\end{eqnarray}
}
\end{lemma}
\begin{proof}
Using the definition of $A_h(t)$ and $A_h(s)$ yields
\begin{eqnarray}
\label{ref5}
&&\Vert (A_h(t)-A_h(s))(-A_h(r))^{-1}u^h\Vert^2\nonumber\\
&=&\left\langle((A_h(t)-A_h(s))(-A_h(r))^{-1}u^h,((A_h(t)-A_h(s))(-A_h(r))^{-1}u^h\right\rangle_H \nonumber\\
&=&\left\langle((A(t)-A(s))(-A_h(r))^{-1}u^h,((A_h(t)-A_h(s))(-A_h(r))^{-1}u^h\right\rangle_H.
\end{eqnarray}
Using Cauchy's Schwartz inequality, the relation $A_h(r)R_h(r)=P_hA(r)$ (see e.g. \cite{Antonio1,Stig2}), \assref{assumption2} (ii) and the boundness of $R_h(r)$  yields
\begin{eqnarray}
\label{ref6a}
&&\Vert (A_h(t)-A_h(s))(-A_h(r))^{-1}u^h\Vert\nonumber\\
&\leq& C\Vert ((A(t)-A(s))(-A_h(r))^{-1}u^h\Vert\nonumber\\
&=&C\Vert ((A(t)-A(s))(-A_h(r))^{-1}P_hu^h\Vert\nonumber\\
&=&C\Vert ((A(t)-A(s))R_h(r)(-A(r))^{-1}u^h\Vert\nonumber\\
&=&C\Vert ((A(t)-A(s))(-A(r))^{-1}(-A(r))R_h(r)(-A(r))^{-1}u^h\Vert\nonumber\\
&\leq&C\vert t-s\vert\Vert (-A(r))R_h(r)(-A(r))^{-1}u^h\Vert.
\end{eqnarray}
Using triangle inequality and \eqref{ritz0b} yields
\begin{eqnarray}
\label{ref6b}
&&\Vert (-A(r))R_h(r)(-A(r))^{-1}u^h\Vert\nonumber\\
&\leq& \Vert (-A(r))R_h(r)(-A(r))^{-1}u^h-A(r)(-A(r))^{-1}u^h\Vert+\Vert A(r)(-A(r))^{-1}u^h\Vert\nonumber\\
&=&\left\Vert A(r)\left(R_h(r)(-A(r))^{-1}u^h-(-A(r))^{-1}u^h\right)\right\Vert+\Vert u^h\Vert\nonumber\\
&=&\Vert R_h(r)(-A(r))^{-1}u^h-(-A(r))^{-1}u^h\Vert_{H^2(\Lambda)}+\Vert u^h\Vert
\nonumber\\
&\leq &C\Vert (-A(r))^{-1}u^h\Vert_{H^2(\Lambda)}+\Vert u^h\Vert\nonumber\\
&\leq& C\Vert u^h\Vert.
\end{eqnarray}
Substituting \eqref{ref6b} in \eqref{ref6a} yields
\begin{eqnarray}
\label{ref6}
\Vert (A_h(t)-A_h(s))(-A_h(r))^{-1}u^h\Vert\leq C\vert t-s\vert  \Vert u^h \Vert.
\end{eqnarray}
This completes the proof of \eqref{ref3}.

 To prove \eqref{ref2}, as in \cite{Antonio2} or \cite{Stig2} we set $V_{r}=\mathcal{D}(-A(r))$, $V_{r}^h=\mathcal{D}(-A_h(r))$, so $V'_{r}=\mathcal{D}\left((-A(r))^{-1}\right)$. Following \cite[(67)]{Antonio2} or \cite{Stig2}, we have
{\small
\begin{eqnarray*}
\left\Vert (-A_h(r))^{-1}\left(A_h(s)-A_h(t)\right)u^h\right\Vert=\sup_{v_h\in V_r^h}\frac{\left\langle \left(A_h(s)-A_h(t)\right)u^h,(-A^*_h(r))^{-1}v_h\right\rangle_H}{\Vert v_h\Vert}
\end{eqnarray*}
}
Using the definition of $A_h(s)$ and $A_h(t)$, it holds that
{\small
\begin{eqnarray}
\left\Vert (-A_h(r))^{-1}\left(A_h(s)-A_h(t)\right)u^h\right\Vert&=&\sup_{v_h\in V_{r}^h}\frac{\left\langle \left(A(s)-A(t)\right)u^h,(-A^*_h(r))^{-1}v_h\right\rangle_H}{\Vert v_h\Vert}\nonumber\\
&=&\sup_{w_h\in V_{r}^h}\frac{\left\langle \left(A(s)-A(t)\right)u^h, w_h\right\rangle_H}{\Vert (-A^*_h(r))w_h\Vert}\nonumber\\
&\leq &C\sup_{w_h\in V_{r}^h}\frac{\left\langle \left(A(s)-A(t)\right)u^h, w_h\right\rangle_H}{\Vert w_h\Vert_{V_{r}}}\nonumber\\
&=&C\left\Vert \left(A(s)-A(t)\right)u^h\right\Vert_{-1}\nonumber\\
&=&C\left\Vert (-A(r))^{-1})\left(A(s)-A(t)\right)u^h\right\Vert\nonumber\\
&\leq& C\vert s-t\vert\,  \Vert u^{h} \Vert ,
\end{eqnarray}
}
where \assref{assumption2} (ii) is used at the last step.
This completes the proof of \eqref{ref2}. The proof of \eqref{ref4} follows from \eqref{ref2} and \eqref{ref3} by interpolation theory.
\end{proof}

\begin{lemma}
\label{lemderiv}
Let \assref{assumption2} be fulfilled. Then for any $u^h\in V_h\cap \mathcal{D}\left(\left(-A(0)\right)^{1-\alpha_2}\right)$ the following estimates hold
\begin{eqnarray}
\label{ba1}
\Vert (-A_h(0))^{-\alpha_1}A_h'(t)(-A_h(0))^{-\alpha_2}u^h\Vert &\leq& C\Vert u^h\Vert,\quad t\in[0,T],\\
\label{ba2}
\Vert (-A_h(0))^{-\alpha_1}A_h''(t)(-A_h(0))^{-\alpha_2}u^h\Vert &\leq& C\Vert u^h\Vert,\quad t\in[0,T],
\end{eqnarray}
where $\alpha_1$ and $\alpha_2$ are defined in \assref{assumption2}.
\end{lemma}

\begin{proof}
Recall that
\begin{eqnarray}
\label{ba3}
A_h'(t)=\lim_{\delta\longrightarrow 0}\frac{A_h(t+\delta)-A_h(t)}{\delta}.
\end{eqnarray}
The proof of \eqref{ba1} is completed by combining \eqref{ref4} and \eqref{ba3}. The proof of \eqref{ba2} follows the same lines as that of \lemref{lemma0a}.
\end{proof}

\begin{remark}
\label{evolutionremark}
From \lemref{lemma0a},  it follows  \cite[Theorem 6.1, Chapter 5]{Pazy} that there exists a unique evolution system $U_h :\Delta(T)\longrightarrow L(H)$,
satisfying \cite[(6.3), Page 149]{Pazy}
\begin{eqnarray}
\label{ref6}
U_h(t,s)=S^h_s(t-s)+\int_s^tS^h_{\tau}(t-\tau)R^h(\tau,s)d\tau,
\end{eqnarray}
where $S^h_s(t):=e^{A_h(s)t}$, $R^h(t,s):=\sum\limits_{m=1}^{\infty}R^h_m(t,s)$, with $R^h_m(t,s)$ satisfying the following  recurrence relation \cite[(6.22), Page 153]{Pazy}
\begin{eqnarray}
R^h_{m+1}&=&\int_s^tR^h_1(t,s)R^h_m(\tau,s)d\tau,\\
R^h_1(t,s)&:=&(A_h(s)-A_h(t))S^h_s(t-s), \quad m\geq 1
\end{eqnarray}
Note also that from \cite[(6.6), Chpater 5, Page 150]{Pazy}, the following identity holds 
\begin{eqnarray}
\label{ref7}
R^h(t,s)=R_1^h(t,s)+\int_s^tR_1^h(t,\tau)R^h(\tau,s)d\tau.
\end{eqnarray} 
The mild solution of \eqref{semi1} is therefore given by
\begin{eqnarray}
\label{mild4}
u^h(t)=U_h(t,0)P_hu_0+\int_0^tU_h(t,s)P_hF(s,u^h(s))ds.
\end{eqnarray}
\end{remark}
\begin{lemma}
\label{pazylemma}
Under \assref{assumption2}, the  evolution system $U_h :\Delta(T)\longrightarrow H$ satisfies the following properties
\begin{itemize}
\item[(i)] $U_h(.,s)\in C^1(]s,T]; L(H))$, $0\leq s\leq T$ and 
\begin{eqnarray}
\frac{\partial U_h}{\partial t}(t,s)=-A_h(t)U_h(t,s), \quad 0\leq s\leq t\leq T,\\
\Vert A_h(t)U_h(t,s)\Vert_{L(H)}\leq \frac{C}{t-s},\quad 0\leq s<t\leq T.
\end{eqnarray}
\item[(ii)] $U_h(t,.)u\in C^1([0,t[; H)$, $0<t\leq T$, $u\in\mathcal{D}(A_h(0))$ and 
\begin{eqnarray}
\frac{\partial U_h}{\partial s}(t,s)u=-U_h(t,s)A_h(s)u,\quad 0\leq s\leq t\leq T\\
 \Vert A_h(t)U_h(t,s)A_h(s)^{-1}\Vert_{L(H)}\leq C, \quad 0\leq s\leq t\leq T.
\end{eqnarray}
\end{itemize}
\end{lemma}
\begin{proof}
The proof is similar  to that of \cite[Theorem 6.1, Chapter 5]{Pazy} by using \eqref{smooth1}, \eqref{smooth2},  \lemref{lemma0a} and \lemref{lemma0}.
\end{proof}

\begin{lemma}
\label{evolutionlemma}
Let  \assref{assumption2} be fulfilled. 
\begin{itemize}
\item[(i)] The following estimates hold
\begin{eqnarray}
\label{reste1}
\Vert R^h_1(t,s)\Vert_{L(H)}\leq C,&&\quad
\Vert R^h_m(t,s)\Vert_{L(H)}\leq \frac{C}{m!}(t-s)^{m-1},\quad m\geq 1,\\
\label{reste2}
\Vert R^h(t,s)\Vert_{L(H)}\leq C,&&\quad \Vert U_h(t,s)\Vert_{L(H)}\leq C,\quad 0\leq s\leq t\leq T.
\end{eqnarray}
\item[(ii)] For any $0\leq\gamma\leq\alpha\leq 1$ and $0\leq s\leq t\leq T$, the following estimates hold
\begin{eqnarray}
\label{ae1}
\Vert (-A_h(r))^{\alpha}U_h(t,s)\Vert_{L(H)}&\leq& C(t-s)^{-\alpha},\quad r\in[0,T],\\
\label{ae3}
\Vert U_h(t,s)(-A_h(r))^{\alpha}\Vert_{L(H)}&\leq& C(t-s)^{-\alpha},\quad r\in[0,T],\\
\label{ae2}
 \Vert (-A_h(r))^{\alpha}U_h(t,s)(-A_h(s))^{-\gamma}\Vert_{L(H)}&\leq& C(t-s)^{\gamma-\alpha}, \quad r\in[0,T].
\end{eqnarray}
\item[(iii)] For any $0\leq s\leq t\leq T$ the following useful estimates hold 
\begin{eqnarray}
\label{hen1}
\Vert \left(U_h(t,s)-\mathbf{I}\right)(-A_h(s))^{-\gamma}\Vert_{L (H)}&\leq& C(t-s)^{\gamma}, \quad 0\leq \gamma\leq 1,\\
\label{hen2}
\Vert \left (-A_h(r))^{-\gamma}(U_h(t,s)-\mathbf{I}\right)\Vert_{L (H)}&\leq& C(t-s)^{\gamma}, \quad 0\leq \gamma\leq 1.
\end{eqnarray}
\end{itemize}
\end{lemma}
\begin{proof}
\begin{itemize}
\item[(i)] The proof of the first estimate of \eqref{reste1} follows the same lines as \cite[Corollary 6.3, Page 153]{Pazy} by using
 \eqref{smooth2}, Lemmas \ref{lemma0} and \ref{lemma0a}. The proof of the second estimate of \eqref{reste1} follows the same lines as \cite[(6.23), Page 153]{Pazy}. The proof of the first estimate of \eqref{reste2} is similar to \cite[(6.26), Page 153]{Pazy} and the proof of the second estimate of \eqref{reste2} is similar to \cite[(6.27), Page 153]{Pazy}.
 \item[(ii)] The estimate of  \eqref{ae1} for $\alpha=1$ is given in \lemref{pazylemma}. The proof of \eqref{ae1} for the case $0\leq\alpha<1$ follows from the integral
 equation \eqref{ref6}. In fact pre-multiplying both sides of \eqref{ref6} by $(-A_h(s))^{\alpha}$, taking the norm in both sides,  using \lemref{lemma0} and \eqref{smooth2} yields
 \begin{eqnarray}
 \label{paz1}
 \Vert (-A_h(r))^{\alpha}U_h(t,s)\Vert_{L(H)}&\leq& \Vert (-A_h(r))^{\alpha}S_s^h(t-s)\Vert_{L(H)}\nonumber\\
 &+&\int_s^t\Vert (-A_h(r))^{\alpha}S_{\tau}^h(t-\tau)\Vert_{L(H)}\Vert R^h(\tau,s)\Vert_{L(H)}d\tau\nonumber\\
 &\leq& C(t-s)^{-\alpha}+C\int_s^t(t-\tau)^{-\alpha}d\tau\nonumber\\
 &\leq& C(t-s)^{-\alpha}.
 \end{eqnarray}
 This proves \eqref{ae1}. The proof of \eqref{ae2} and \eqref{ae3} are similar to that of \eqref{ae1}.
 \item[(iii)] From \eqref{ref6}, it holds that
 \begin{eqnarray}
 \label{paz2}
 (U_h(t,s)-\mathbf{I})(-A_h(r))^{-\gamma}&=&(-A_h(s))^{-\gamma}\left(e^{A(s)(t-s)}-\mathbf{I}\right)\nonumber\\
 &+&\int_s^tS^h_{\tau}(t-\tau)R^h(\tau,s)(-A_h(s))^{-\gamma}d\tau.
 \end{eqnarray}
 Taking the norm in both sides of \eqref{paz2}, using \eqref{smooth1},  the boundness of $(-A_h(r))^{-\gamma}$ and \lemref{evolutionlemma} (i) yields
 \begin{eqnarray}
 \Vert (U_h(t,s)-\mathbf{I})(-A_h(s))^{-\gamma}\Vert_{L(H)}&=&C(t-s)^{\gamma}+C\int_s^td\tau\leq C(t-s)^{\gamma}.\nonumber
 \end{eqnarray}
 This completes the proof of \eqref{hen1}. The proof of \eqref{hen2} is similar to that of \eqref{hen1}.
\end{itemize}
\end{proof}

The following space regularity of the semi-discrete problem \eqref{semi1} will be useful in our convergence analysis.
\begin{lemma} 
\label{regularitylemma}
Let  \assref{assumption2} (i)-(ii), \assref{assumption1}  and \assref{assumption3}  be fulfilled with the corresponding  $0\leq \beta<2$. Then for all $\gamma\in[0,\beta]$ and $\alpha\in[0,2)$ the following estimates hold
\begin{eqnarray}
\label{regular1}
\Vert (-A_h(r))^{\gamma/2}u^h(t)\Vert&\leq& C,\quad\hspace{2cm} 0\leq r,t\leq T,\\
\label{regular1a}
\Vert (-A_h(0))^{\alpha/2}u^h(t)\Vert&\leq& Ct^{\beta/2-\alpha/2},\quad t\in[0, T], \quad \beta\in[0,2],
\end{eqnarray}
\end{lemma}
\begin{proof}
 We first show that 
 \begin{eqnarray}
 \label{regular8a}
 \Vert u^h(t)\Vert\leq C,\quad t\in[0,T].
 \end{eqnarray}
  Taking the norm in both side of \eqref{mild4} and using the  triangle inequality yields
 {\small
\begin{eqnarray}
\label{regular3}
\Vert u^h(t)\Vert\leq \Vert U_h(t,0)P_hu_0\Vert+\left\Vert\int_0^tU_h(t,s)P_hF(s,u^h(s))ds\right\Vert ds
:=I_0+I_1.
\end{eqnarray}
}
Using \lemref{evolutionlemma} (i) and the uniformly boundedness of $P_h$, it holds that
\begin{eqnarray}
\label{regular4}
I_0\leq \Vert u_0\Vert\leq C.
\end{eqnarray}
Using  \lemref{evolutionlemma} (i), \assref{assumption3} and the uniformly boundedness of $P_h$, it holds that
\begin{eqnarray}
\label{regular5}
I_1&\leq& \int_0^t\Vert U_h(t,s)P_hF(s,u^h(s))\Vert\leq C\int_0^t\left(C+\Vert u^h(s)\Vert\right)ds\nonumber\\
&\leq& C+C\int_0^t\Vert u^h(s)\Vert ds.
\end{eqnarray}
Substituting \eqref{regular5} and \eqref{regular4} in \eqref{regular3} yields
\begin{eqnarray}
\label{regular8}
\Vert u^h(t)\Vert\leq C+C\int_0^t\Vert u^h(s)\Vert ds.
\end{eqnarray}
Applying the continuous Gronwall's lemma to \eqref{regular8} completes the proof of \eqref{regular8a}.
Let us now prove \eqref{regular1}. 
Pre-multiplying \eqref{mild4} by $(-A_h(r))^{\gamma/2}$, taking  the norm in both sides  and using triangle inequality  yields
\begin{eqnarray}
\label{regular9}
\left\Vert (-A_h(r))^{\gamma/2}u^h(t)\right\Vert&\leq& \left\Vert (-A_h(r))^{\gamma/2}U_h(t,0)P_hu_0\right\Vert_{L(H)}\nonumber\\
&+&\int_0^t\left\Vert (-A_h(r))^{\gamma/2}U_h(t,s)P_hF(s,u^h(s))\right\Vert ds
\nonumber\\
&=:&II_0+II_1.
\end{eqnarray}
Inserting $(-A_h(0))^{-\gamma/2}(-A_h(0))^{\gamma/2}$, using  \lemref{evolutionlemma} (ii) and \lemref{lemma0}, it holds that
{\small
\begin{eqnarray}
\label{premier}
II_0\leq \Vert (-A_h(r))^{\gamma/2}U_h(t,0)(-A_h(0))^{-\gamma/2}\Vert_{L(H)}\Vert (-A_h(0))^{\gamma/2}u_0\Vert\leq C.
\end{eqnarray}
}
Using  \lemref{lemma0}, \lemref{evolutionlemma} (ii), \assref{assumption3} and \eqref{regular8a} yields
\begin{eqnarray}
\label{deuxieme}
II_1&\leq& C\left(\int_0^t\left\Vert(-A_h(s))^{\gamma/2}U_h(t,s)\right\Vert_{L(H)}ds\right)\sup_{r\in[0,T]}\left\Vert F\left(r,u^h(r)\right)\right\Vert \nonumber\\
&\leq& C\sup_{s\in[0,T]}\left(1+\Vert u^h(s)\Vert\right)\int_0^t(t-s)^{-\gamma/2}ds\leq C.
\end{eqnarray}
Substituting \eqref{deuxieme} and \eqref{premier} in \eqref{regular9} completes the proof of \eqref{regular1}. 
The proof of \eqref{regular1a} is similar to that of \eqref{regular1}.  This completes the proof of \lemref{regularitylemma}.
\end{proof}
Let us consider the following deterministic  problem: find $w\in V$ such that
\begin{eqnarray}
\label{determ1}
w'=A(t)w,\quad w(\tau)=v,\quad t\in(\tau,T].
\end{eqnarray}
The corresponding semi-discrete problem in space is: find $w_h\in V_h$ such that
\begin{eqnarray}
\label{determ2}
w_h'(t)=A_h(t)w_h,\quad w_h(\tau)=P_hv,\quad t\in(\tau,T],\quad \tau\geq 0.
\end{eqnarray}
Let us define the operator
\begin{eqnarray}
T_h(t,\tau):=U(t,\tau)-U_h(t,\tau)P_h,
\end{eqnarray}
so that $w(t)-w_h(t)=T_h(t,\tau)v$. The following lemma will be useful in our convergence analysis. 
\begin{lemma}
\label{spaceerrorlemma}
Let $r\in[0,2]$ and  $\gamma\leq r$. Let  \assref{assumption2} be fulfilled.  Then the following error estimate holds for the semi-discrete approximation  \eqref{determ2}
\begin{eqnarray}
\label{er0}
\Vert w(t)-w_h(t)\Vert=\Vert T_h(t,\tau)v\Vert\leq Ch^r(t-\tau)^{-(r-\gamma)/2}\Vert v\Vert_{\gamma},
\end{eqnarray}
for any $v\in \mathcal{D}\left(\left(-A(0)\right)^{\gamma/2}\right)$.
\end{lemma}
\begin{proof}
As in \cite[(3.5)]{Antonio1} or \cite{Stig2}, we  set
\begin{eqnarray}
\label{espa0}
w_h(t)-w(t)&=&\left(w_h(t)-R_h(t)w(t)\right)+\left(R_h(t)w(t)-w(t)\right)\nonumber\\
&\equiv& \theta(t)+\rho(t).
\end{eqnarray}
Using the definition of $R_h(t)$ and $P_h$, we can prove exactly as in \cite{Stig2,Antonio1} that
\begin{eqnarray}
\label{espacetamb}
A_h(t)R_h(t)=P_hA(t),\quad t\in[0,T].
\end{eqnarray}
One can easily compute the following derivatives 
\begin{eqnarray}
\label{espace1a}
\theta_t&=&A_h(t)w_h(t)-R_h'(t)w(t)-R_h(t)A(t)w(t),\\
\label{espace1b}
D_t\rho&=&R_h'(t)w(t)+R_h(t)A(t)w(t)-A(t)w(t).
\end{eqnarray}
Endowing $V$ and the linear subspace $V_h$ with the $\Vert .\Vert_{H^1(\Lambda)}$ norm, it follows from \eqref{ritz2} that $R_h(t)\in L(V, V_h)$ for all $t\in [0, T]$. By the definition of the differential operator, it follows that $R_h'(t)\in L(V, V_h)$ for all $t\in[0, T]$. Hence $P_hR_h'(t)=R_h'(t)$ for all $t\in[0,T]$ 
and it follows from \eqref{espace1b} that 
\begin{eqnarray}
\label{espace1c}
P_hD_t\rho=R_h'(t)w(t)+R_h(t)A(t)w(t)-P_hA(t)w(t).
\end{eqnarray}
Adding and subtracting $P_hA(t)w(t)$ in \eqref{espace1a} and using \eqref{espacetamb}, it follows that $\theta$ satisfies the following equation
\begin{eqnarray}
\label{espa1}
\theta_t=A_h(t)\theta-P_hD_t\rho,\quad t\in(\tau,T],
\end{eqnarray}
Since $\{A_h(t)\}_{t\in[0,T]}$ generates an evolution system $\{U_h(t,s)\}_{0\leq s\leq t\leq T}$, it holds that
\begin{eqnarray}
\label{espa2}
\theta(t)=U_h(t,\tau)\theta(\tau)-\int_{\tau}^tU_h(t,s)P_hD_s\rho(s)ds.
\end{eqnarray}
Splitting the  integral part of \eqref{espa2} into two intervals and integrating by parts over the first interval yields
\begin{eqnarray}
\label{espa3}
\theta(t)&=& U_h(t,\tau)\theta(\tau)+U_h(t,\tau)P_h\rho(\tau)-U_h\left(t,(t+\tau)/2\right)P_h\rho\left((t+\tau)/2\right)\nonumber\\
&+&\int_{\tau}^{(t+\tau)/2}\frac{\partial}{\partial s}\left(U_h(t,s)\right)P_h\rho(s)ds-\int_{(t+\tau)/2}^tU_h(t,s)P_hD_s\rho(s)ds.
\end{eqnarray}
Using the expression of $\theta(\tau)$, $\rho(\tau)$ and the fact that $u_h(\tau)=P_hv$, it holds that
\begin{eqnarray}
\label{espa4}
\theta(\tau)+P_h\rho(\tau)=0.
\end{eqnarray}
Using \eqref{espa4} reduces  \eqref{espa3}  to
\begin{eqnarray}
\label{espa5}
\theta(t)&=& -U_h(t,s)P_h\rho((t+\tau)/2)+\int_{\tau}^{(t+\tau)/2}\frac{\partial}{\partial s}\left(U_h(t,s)
\right)P_h\rho(s)ds\nonumber\\
&-&\int_{(t+\tau)/2}^tU_h(t,s)P_hD_s\rho(s)ds.
\end{eqnarray}
Taking the norm in both sides of \eqref{espa5}, using the uniformly boundedness of $P_h$, \eqref{smooth2}, Lemma  \ref{lemma0a} and Lemma \ref{evolutionlemma} (i) yields
{\small
\begin{eqnarray}
\label{espa6}
\Vert\theta(t)\Vert&\leq& C\Vert \rho((t+\tau)/2)\Vert+\int_{\tau}^{(t+\tau)/2}\left\Vert U_h(t,s)A_h(s)\right\Vert_{L(H)}\Vert \rho(s)\Vert ds+\int_{(t+\tau)/2}^t\Vert D_s\rho(s)\Vert ds\nonumber\\
&\leq& C\Vert \rho((t+\tau)/2)\Vert+\int_{\tau}^{(t+\tau)/2}(t-s)^{-1}\Vert \rho(s)\Vert ds+\int_{(t+\tau)/2}^t\Vert D_s\rho(s)\Vert ds.
\end{eqnarray}
}
Using \eqref{ritz2}, it holds that
\begin{eqnarray}
\label{espa7}
\Vert \rho(s)\Vert\leq  Ch^r\Vert w(s)\Vert_r.
\end{eqnarray}
Note that the solution of \eqref{determ1} is represented as follows.
\begin{eqnarray}
\label{encore1}
w(s)=U(s,\tau)v,\quad s\geq \tau.
\end{eqnarray}
Pre-multiplying both sides of \eqref{encore1} by $(-A(s))^{r/2}$, inserting an appropriate power of  $-A(\tau)$,  using \lemref{evolutionlemma} (ii) and \cite[Lemma 1]{Antjd1} yields
\begin{eqnarray}
\label{encore2}
\Vert (-A(t))^{r/2}w(s)\Vert&\leq& \Vert (-A(s))^{r/2}U(s,\tau)(-A(\tau))^{-\gamma/2}\Vert_{L(H)}\Vert (-A(\tau))^{\gamma/2}v\Vert\nonumber\\
&\leq& C(s-\tau)^{-(r-\gamma)/2}\Vert (-A(\tau))^{\gamma/2}v\Vert\nonumber\\
&\leq& C(s-\tau)^{-(r-\gamma)/2}\Vert v\Vert_{\gamma}.
\end{eqnarray}
Therefore it holds that
\begin{eqnarray}
\label{espa8}
\Vert w(s)\Vert_r\leq C(s-\tau)^{-(r-\gamma)/2}\Vert v\Vert_{\gamma}, \quad 0\leq \gamma\leq r\leq 2,\quad \tau<s.
\end{eqnarray}
Substituting \eqref{espa8} in \eqref{espa7} yields
\begin{eqnarray}
\label{espa8a}
\Vert \rho(s)\Vert_r\leq Ch^r(s-\tau)^{-(r-\gamma)/2}\Vert v\Vert_{\gamma}.
\end{eqnarray}
Using \eqref{ritz3}, it holds that
\begin{eqnarray}
\label{espa13}
\Vert D_s\rho(s)\Vert \leq Ch^r(\Vert w(s)\Vert_r+\Vert D_sw(s)\Vert_r).
\end{eqnarray}
Taking the derivative with respect to $s$ in both sides of \eqref{encore1} yields
\begin{eqnarray}
\label{encore3}
D_sw(s)=A(s)U(s,\tau)v.
\end{eqnarray}
As for  \eqref{encore2}, pre-multiplying both sides of \eqref{encore3} by $(-A(s))^{r/2}$, inserting $(-A(\tau))^{-\gamma/2}(-A(\tau))^{\gamma/2}$ and 
using \lemref{evolutionlemma} (ii)  yields
\begin{eqnarray}
\label{encore4a}
\Vert D_sw(s)\Vert_r\leq C(s-\tau)^{-1-(r-\gamma)/2}\Vert v\Vert_{\gamma}.
\end{eqnarray}
Substituting \eqref{espa8} and  \eqref{encore4a} in \eqref{espa13} yields
\begin{eqnarray}
\label{espa15}
\Vert D_s\rho(s)\Vert&\leq& Ch^r\left((s-\tau)^{-(r-\gamma)/2}\Vert v\Vert_{\gamma}+(s-\tau)^{-1-(r-\gamma)/2}\Vert v\Vert_{\gamma}\right)\nonumber\\
&\leq& Ch^r(s-\tau)^{-1-(r-\gamma)/2}\Vert v\Vert_{\gamma}.
\end{eqnarray}
Substituting \eqref{espa8a} and \eqref{espa15} in \eqref{espa6} yields 
\begin{eqnarray}
\label{espa17}
\Vert\theta(t)\Vert&\leq& Ch^r(t-\tau)^{-(r-\gamma)/2}\Vert v\Vert_{\gamma}\nonumber\\
&+&Ch^r\int_{\tau}^{(t+\tau)/2}(t-s)^{-1}(s-\tau)^{-(r-\gamma)/2}\Vert v\Vert_{\gamma}ds\nonumber\\
&+&Ch^r\int_{(t+\tau)/2}^t(s-\tau)^{-1-(r-\gamma)/2}\Vert v\Vert_{\gamma}ds.
\end{eqnarray}
Using the estimate 
{\small
\begin{eqnarray}
\label{espa18}
\int_{\tau}^{(t+\tau)/2}(t-s)^{-1}(s-\tau)^{-(r-\gamma)/2}ds+\int_{(t+\tau)/2}^t(s-\tau)^{-1-(r-\gamma)/2}ds\leq C(t-\tau)^{-(r-\gamma)/2},\nonumber
\end{eqnarray}
}
it follows  that
\begin{eqnarray}
\label{espa20}
\Vert \theta(t)\Vert\leq  Ch^r(t-\tau)^{-(r-\gamma)/2}\Vert v\Vert_{\gamma}.
\end{eqnarray}
Substituting \eqref{espa20} and \eqref{espa13} in \eqref{espa0} yields
\begin{eqnarray}
\Vert w(t)-w_h(t)\Vert\leq \Vert\theta(t)\Vert+\Vert \rho(t)\Vert\leq Ch^r(t-\tau)^{-(r-\gamma)/2}\Vert v\Vert_{\gamma}.
\end{eqnarray}
This completes the proof of \lemref{spaceerrorlemma}.
\end{proof}
\begin{remark}
\label{remaksemi}
 \lemref{spaceerrorlemma} generalizes \cite[Lemma 3.1]{Antonio1} to time dependent problems. 
 It also generalises  \cite[Lemmas 3.2 and 3.3]{Luskin} and \cite[Theorems 3 and 4]{Thomee1} to more general boundary conditions than only Dirichlet boundary conditions.
 Note that the fact that the solution vanishes at the boundary is fundamental in the proof of \cite[Lemmas 3.2 and 3.3]{Luskin} and \cite[Theorems 3 and 4]{Thomee1}, where authors used energy estimates arguments.
\end{remark}

The following theorem gives the space convergence error of the semi-discrete solution in space toward the exact solution. 
It is fundamental in the proof of the convergence of the fully discrete scheme.
\begin{theorem}
\label{proposition2}
Let  \assref{assumption2}, \assref{assumption1} and \assref{assumption3}  be fulfilled. Let $u(t)$ and $u^h(t)$ be the mild solution of \eqref{model} and \eqref{semi1} respectively.
\begin{itemize}
\item[(i)] If $0<\beta<2$, 
then the following error estimate holds 
\begin{eqnarray}
\label{time1}
\Vert u(t)-u^h(t)\Vert\leq Ch^{\beta},\quad 0\leq t\leq T.
\end{eqnarray}
\item[(ii)] If $\beta=2$, then the following error estimate holds
\begin{eqnarray}
\label{time2}
\Vert u(t)-u^h(t)\Vert\leq Ch^{2}\left(1+\max\left(0,\ln(t/h^2)\right)\right),\quad 0< t\leq T.
\end{eqnarray}
\item[(iii)] If $\beta=2$ and if further \assref{assumption4} is fulfilled, then the following error estimate holds
\begin{eqnarray}
\Vert u(t)-u^h(t)\Vert\leq Ch^2,\quad 0\leq t\leq T.
\end{eqnarray}
\end{itemize}
\end{theorem}

\begin{proof} 
Subtracting \eqref{mild4} form \eqref{mild0},  taking the  norm  and using triangle inequality yields
\begin{eqnarray}
\label{estiI}
\Vert u(t)-u^h(t)\Vert&\leq& \left\Vert U(t,0)u_0-U_h(t,0)P_hu_0\right\Vert\nonumber\\
&+&\left\Vert\int_0^{t}\left[U(t,s)F\left(s,u(s)\right)-U_h(t,s)P_hF\left(s,u^h(s)\right)\right]
 ds\right\Vert\nonumber\\
 & =:&III_0+III_1.
\end{eqnarray}
Using \lemref{spaceerrorlemma} with $r=\gamma=\beta$  yields
\begin{eqnarray}
\label{jour1}
III_0\leq Ch^{\beta}\Vert u_0\Vert\leq Ch^{\beta}.
\end{eqnarray}
Using \lemref{spaceerrorlemma} with $r=\beta$ (with $\beta<2$), $\gamma=0$, \assref{assumption3},  \lemref{regularitylemma} and \lemref{evolutionlemma} yields
\begin{eqnarray}
\label{estiI2}
III_1&\leq& \int_0^t\left\Vert U(t,s)F\left(s,u(s)\right)-U(t,s)F\left(s,u^h(s)\right)\right\Vert ds\nonumber\\
&+&\int_0^t\left\Vert U(t,s)F\left(s,u^h(s)\right)-U_h(t,s)P_hF\left(s,u^h(s)\right)\right\Vert ds\nonumber\\
&\leq& C\int_0^t\left\Vert u(s)-u^h(s)\right\Vert_{L^2(\Omega,H)}ds+Ch^{\beta}\int_0^t(t-s)^{-\beta/2}ds\nonumber\\
&\leq& Ch^{\beta}+C\int_0^{t}\left\Vert u(s)-u^h(s)\right\Vert ds.
\end{eqnarray}
Substituting \eqref{estiI2} and \eqref{jour1} in \eqref{estiI} yields
\begin{eqnarray}
\label{lat1}
\left\Vert u(t)-u^h(t)\right\Vert\leq Ch^{\beta}+C\int_0^{t}\left\Vert u(s)-u^h(s)\right\Vert ds.
\end{eqnarray}
Applying the continuous Gronwall's lemma to \eqref{lat1} prove \eqref{time1}.
The proof of \eqref{time2} is straightforward. This completes the proof of  \propref{proposition2}.
\end{proof}
The following lemma extends some results in \cite{Luskin} (see e.g. \cite[Lemma 2.4, (2.8)]{Luskin} and \cite[Lemma 2.6]{Luskin}) 
to the case of fully semilinear problem. It also extends \cite[Lemma 3.7]{Antjd2} to the case of non-autonomous problems.
\begin{lemma}
\label{ancien}
Let \assref{assumption1} (with $0<\beta<2$), \assref{assumption2},  \assref{assumption3} and \assref{assumption4} be fulfilled. 
\begin{itemize}
\item[(i)] The following estimate holds
\begin{eqnarray}
 \Vert D_tu^h(t)\Vert\leq Ct^{-1+\beta/2},\quad t\in[0,T].
\end{eqnarray}
\item[(ii)] For any $\alpha\in(0, \beta)$, the following estimate holds
\begin{eqnarray}
\left\Vert (-A_h(0))^{\alpha/2}D_tu^h(t)\right\Vert\leq Ct^{-1-\alpha/2+\beta/2},\quad t\in(0,T].
\end{eqnarray}
\item[(iii)] The following holds
\begin{eqnarray}
\Vert D^2_tu^h(t)\Vert\leq Ct^{-2+\beta/2},\quad t\in(0,T].
\end{eqnarray}
\end{itemize}
\end{lemma}
\begin{proof}
As in the proof of \cite[Theorem 5.2]{Stig2} or \cite[Lemma 3.7]{Antjd2}, we  set $v^h(t)=tD_tu^h(t)$, it follows that $D_tv^h(t)$ satisfies the following equation
\begin{eqnarray}
D_tv^h(t)&=&A_h(t)v^h(t)+D_tu^h(t)+tA_h'(t)u^h(t)+tP_h\frac{\partial F}{\partial t}\left(t, u^h(t)\right)\nonumber\\
&+&tP_h\frac{\partial F}{\partial u}\left(t,u^h(t)\right)v^h(t).
\end{eqnarray}
Therefore by the Duhammel's principle, it holds that
\begin{eqnarray}
\label{deri1}
v^h(t)&=&\int_0^tU_h(t,s)\left[D_tu^h(t)+sA_h'(s)u^h(s)+sP_h\frac{\partial F}{\partial s}\left(s, u^h(s)\right)\right]ds\nonumber\\
&+&\int_0^tsU_h(t,s)P_h\frac{\partial F}{\partial u}\left(s,u^h(s)\right)v^h(s)ds.
\end{eqnarray}
Taking the norm in both sides of \eqref{deri1}, using \assref{assumption3}  and \lemref{pazylemma} yields
{\small
\begin{eqnarray}
\label{inter1}
\Vert v^h(t)\Vert&\leq& \int_0^t\left\Vert U_h(t,s)D_su^h(s)\right\Vert ds+\int_0^ts\left\Vert U_h(t,s)A_h'(s)u^h(s)\right\Vert ds\nonumber\\
&+&\int_0^ts\left\Vert U_h(t,s)P_h\frac{\partial F}{\partial s}\left(s, u^h(s)\right)\right\Vert ds+\int_0^ts\left\Vert U_h(t,s) P_h\frac{\partial F}{\partial u}\left(s,u^h(s)\right)\right\Vert_{L(H)}\Vert v^h(s)\Vert ds\nonumber\\
&\leq&  \int_0^t\left\Vert U_h(t,s)D_su^h(s)\right\Vert ds+\int_0^ts\left\Vert U_h(t,s)A_h'(s)u^h(s)\right\Vert ds+Ct^2\nonumber\\
&+&C\int_0^t\Vert v^h(s)\Vert ds.
\end{eqnarray}
}
Using \lemref{lemderiv} and \lemref{regularitylemma} yields
\begin{eqnarray}
\label{inter2}
&&\int_0^ts\left\Vert U_h(t,s)A_h'(s)u^h(s)\right\Vert ds\nonumber\\
&\leq&\int_0^ts\left\Vert U_h(t,s)\left(-A_h(0)\right)^{1-\beta/2}\right\Vert_{L(H)}\left\Vert\left(-A_h(0)\right)^{-1+\beta/2} A_h'(s)(-A_h(0))^{-\beta/2}\right\Vert_{L(H)}\nonumber\\
&&\left\Vert(-A_h(0))^{\beta/2}u^h(s)\right\Vert ds\nonumber\\
&\leq& Ct\int_0^t(t-s)^{-1+\beta/2}\left\Vert (-A_h(0))^{\beta/2}u^h(s)\right\Vert ds\nonumber\\
&\leq& Ct\int_0^t(t-s)^{-1+\beta/2}ds\leq Ct^{1+\beta/2}.
\end{eqnarray}
Using \lemref{pazylemma} and \lemref{regularitylemma}, we obtain
{\small
\begin{eqnarray}
\label{inter3}
&&\left\Vert U_h(t,s)D_su^h(s)\right\Vert\nonumber\\
&\leq& \left\Vert U_h(t,s)A_h(s)u^h(s)\right\Vert+\Vert U_h(t,s)P_hF(u^h(s))\Vert\nonumber\\
&\leq& \left\Vert U_h(t,s)(-A_h(0))^{1-\beta/2}\right\Vert_{L(H)}\left\Vert (-A_h(0))^{\beta/2}u^h(s)\right\Vert+\Vert U_h(t,s)\Vert_{L(H)}\Vert P_hF(u^h(s))\Vert\nonumber\\
&\leq&C(t-s)^{-1+\beta/2}\Vert u_0\Vert_{\beta}+C\Vert u_0\Vert\nonumber\\
&\leq& C(t-s)^{-1+\beta/2}.
\end{eqnarray}
}
Substituting \eqref{inter3} and \eqref{inter2} in \eqref{inter1} yields
\begin{eqnarray}
\label{inter4}
\Vert v^h(t)\Vert&\leq& C\int_0^t(t-s)^{-1+\beta/2}ds+Ct^2+C\int_0^t\Vert v^h(s)\Vert ds\nonumber\\
&\leq& Ct^{\beta/2}+C\int_0^t\Vert v^h(s)\Vert ds.
\end{eqnarray}
Applying the continuous Gronwall's lemma to \eqref{inter4} yields
\begin{eqnarray}
\Vert v^h(t)\Vert \leq t^{\beta/2}.
\end{eqnarray}
Therefore we have 
\begin{eqnarray}
\Vert D_tu^h(t)\Vert\leq Ct^{-1+\beta/2}.
\end{eqnarray}
Let us now prove (ii). 
It follows from \eqref{deri1} that 
\begin{eqnarray}
\label{deri2}
D_tu^h(t)&=&t^{-1}\int_0^tU_h(t,s)\left[D_su^h(s)+sA_h'(s)u^h(s)+sP_h\frac{\partial F}{\partial s}\left(s, u^h(s)\right)\right]ds\nonumber\\
&+&t^{-1}\int_0^tU_h(t,s)sP_h\frac{\partial F}{\partial u}\left(s,u^h(s)\right)D_su^h(s)ds.
\end{eqnarray}
Pre-multiplying both sides of \eqref{deri2} by $(-A_h(0))^{\alpha/2}$ yields
\begin{eqnarray}
\label{deri3}
&&(-A_h(0))^{\alpha/2}D_tu^h(t)\nonumber\\
&=&t^{-1}\int_0^t(-A_h(0))^{\alpha/2}U_h(t,s)\left[D_su^h(s)+sA_h'(s)u^h(s)P_h\frac{\partial F}{\partial s}\left(s,u^h(s)\right)\right]ds\nonumber\\
&+&t^{-1}\int_0^ts\left(-A_h(0)\right)^{\alpha/2}U_h(t,s)P_h\frac{\partial F}{\partial u}\left(s,u^h(s)\right)D_su^h(s)ds.
\end{eqnarray}
Taking the norm in both sides of \eqref{deri3} yields
{\small
\begin{eqnarray}
\label{inter5}
&&\left\Vert (-A_h(0))^{\alpha/2}D_tu^h(t)\right\Vert\nonumber\\
&\leq&t^{-1}\int_0^t(t-s)^{-\alpha/2}\left[\Vert D_su^h(s)\Vert+s\left\Vert \frac{\partial F}{\partial s}\left(s,u^h(s)\right)\right\Vert\right]ds\nonumber\\
&+&t^{-1}\int_0^t(t-s)^{-\alpha/2}s\left\Vert P_h\frac{\partial F}{\partial u}\left(s,u^h(s)\right)\right\Vert_{L(H)}\Vert D_su^h(s)\Vert ds\nonumber\\
&+&t^{-1}\int_0^ts\left\Vert\left(-A_h(0)\right)^{\alpha/2}U_h(t,s)A_h'(s)u^h(s)\right\Vert ds\nonumber\\
&\leq& Ct^{-1}\int_0^t(t-s)^{-\alpha/2}\left[s^{-1+\beta/2}+s\right]ds+t^{-1}\int_0^ts\left\Vert(-A_h(0)^{\alpha/2}U_h(t,s)A_h'(s)u^h(s)\right\Vert ds\nonumber\\
&\leq &Ct^{-1}\int_0^t(t-s)^{-\alpha/2}s^{-1+\beta/2}ds+\int_0^t\Vert(-A_h(0)^{\alpha/2}U_h(t,s)A_h'(s)u^h(s)\Vert ds.
\end{eqnarray}
}
Using \lemref{lemderiv} and \lemref{regularitylemma}, it holds that
\begin{eqnarray}
\label{inter6}
&&\int_0^t\left\Vert (-A_h(0))^{\alpha/2}U_h(t,s)A_h'(s)u^h(s)\right\Vert ds\nonumber\\
&\leq& \int_0^t\left\Vert (-A_h(0))^{\alpha/2}U_h(t,s)(-A_h(0))^{\epsilon}\right\Vert_{L\left(\mathcal{D}\left(-A(0)\right)^{\epsilon}, H\right)}\nonumber\\
&&\times\left\Vert (-A_h(0))^{-\epsilon}A_h'(s)(-A_h(0))^{-1+\epsilon}(-A_h(0))^{1-\epsilon}u^h(s)\right\Vert ds\nonumber\\
&\leq& C\int_0^t(t-s)^{-\alpha/2}\left\Vert (-A_h(0))^{1-\epsilon}u^h(s)\right\Vert ds\nonumber\\
&\leq& C\int_0^t(t-s)^{-\alpha/2-\epsilon}s^{\beta/2-1+\epsilon}ds\nonumber\\
&\leq & Ct^{\beta/2-\alpha/2}.
\end{eqnarray}
Substituting \eqref{inter6} dans \eqref{inter5} yields
\begin{eqnarray}
\label{inter7}
\left\Vert (-A_h(0))^{\alpha/2}D_tu^h(t)\right\Vert&\leq& Ct^{-1}\int_0^t(t-s)^{-\alpha/2}s^{-1+\beta/2}ds+Ct^{-\alpha/2-\epsilon}\nonumber\\
&\leq& Ct^{-1-\alpha/2+\beta/2}.
\end{eqnarray}
This completes the proof of (ii).
To prove (iii), as in \cite[Lemma 3.7]{Antjd2} we set $w^h(t)=tD^2_tu^h(t)$. Taking the derivative with respect to $t$ in both sides of \eqref{semi1} yields 
\begin{eqnarray}
\label{deri4}
D^2_tu^h(t)&=&A_h'(t)u^h(t)+A_h(t)D_tu^h(t)+P_h\frac{\partial F}{\partial t}\left(t,u^h(t)\right)\nonumber\\
&+&P_h\frac{\partial F}{\partial u}\left(t,u^h(t)\right)D_tu^h(t).
\end{eqnarray} 
Taking the derivative with respect to $t$  in both side of \eqref{deri4} yields
{\small
\begin{eqnarray}
\label{deri5}
D^3_tu^h(t)&=&A_h''(t)u^h(t)+2A_h'(t)D_tu^h(t)+A_h(t)D^2_tu^h(t)\nonumber\\
&+&P_h\frac{\partial^2F}{\partial t^2}\left(t,u^h(t)\right)D_tu^h(t)+2P_h\frac{\partial^2F}{\partial t\partial u}\left(t,u^h(t)\right)D_tu^h(t)\nonumber\\
&+&P_h\frac{\partial^2F}{\partial t\partial u}\left(t,u^h(t)\right)D_tu^h(t)+P_h\frac{\partial^2F}{\partial u^2}\left(t,u^h(t)\right)\left(D_tu^h(t), D_tu^h(t)\right).
\end{eqnarray}
}
Using \eqref{deri5} and \eqref{deri4} and rearranging  yields
{\small
\begin{eqnarray}
\label{deri6}
D_tw^h(t)&=&D_t^2u^h(t)+tD_t^3u^h(t)\nonumber\\
&=&A_h(t)w^h(t)+A_h'(t)u^h(t)+A_h(t)D_tu^h(t)+P_h\frac{\partial F}{\partial t}\left(t,u^h(t)\right)\nonumber\\
&+&P_h\frac{\partial F}{\partial u}\left(t,u^h(t)\right)D_tu^h(t)+tA_h''(t)u^h(t)+2tA_h'(t)D_tu^h(t)\nonumber\\
&+&tP_h\frac{\partial^2F}{\partial t^2}\left(t,u^h(t)\right)D_tu^h(t)+2tP_h\frac{\partial^2F}{\partial t\partial u}\left(t,u^h(t)\right)D_tu^h(t)\nonumber\\
&+&tP_h\frac{\partial^2F}{\partial t\partial u}\left(t,u^h(t)\right)D_tu^h(t)+tP_h\frac{\partial^2F}{\partial u^2}\left(t,u^h(t)\right)\left(D_tu^h(t),D_tu^h(t)\right).
\end{eqnarray}
}
By the Duhammel's principle, it follows  from \eqref{deri6} that
{\small
\begin{eqnarray}
\label{deri7}
w^h(t)&=&\int_0^tU_h(t,s)\left[A_h'(s)u^h(s)+A_h(s)D_su^h(s)+sA_h''(s)u^h(s)
+2sA_h'(s)D_su^h(s)\right]ds\nonumber\\
&+&\int_0^tU_h(t,s)\left[P_h\frac{\partial F}{\partial s}\left(s,u^h(s)\right)+P_h\frac{\partial F}{\partial u}\left(s,u^h(s)\right)D_su^h(s)\right]ds\nonumber\\
&+&\int_0^tU_h(t,s)sP_h\left[sP_h\frac{\partial^2F}{\partial s^2}\left(s,u^h(s)\right)D_su^h(s)+3\frac{\partial^2F}{\partial s\partial u}\left(s,u^h(s)\right)D_su^h(s)\right]ds\nonumber\\
&+&\int_0^tsU_h(t,s)\frac{\partial^2F}{\partial u^2}\left(s,u^h(s)\right)\left(D_su^h(s), D_su^h(s)\right)ds.
\end{eqnarray}
}
Taking the norm in both sides of \eqref{deri7} yields
{\small
\begin{eqnarray}
\label{deri8}
\Vert w^h(t)\Vert&\leq& \int_0^t\left\Vert U_h(t,s)A_h'(s)u^h(s)\right\Vert ds+\int_0^t\left\Vert U_h(t,s)A_h(s)D_su^h(s)\right\Vert ds\nonumber\\
&+&t\int_0^t\left\Vert U_h(t,s)A_h''(s)u^h(s)\right\Vert ds+2t\int_0^t\left\Vert U_h(t,s)D_su^h(s)\right\Vert ds\nonumber\\
&+&C\int_0^t\Vert D_su^h(s)\Vert ds+C\int_0^ts\Vert D_su^h(s)\Vert ds+C\int_0^ts\Vert D_su^h(s)\Vert^2ds.
\end{eqnarray}
}
Using  \lemref{lemderiv}, \lemref{pazylemma} and \lemref{regularitylemma} yields
\begin{eqnarray}
\label{inter8}
\int_0^t\left\Vert U_h(t,s)A_h'(s)u^h(s)\right\Vert ds\leq Ct^{\beta/2}.
\end{eqnarray}
Using (ii) and  \lemref{regularitylemma} yields 
\begin{eqnarray}
\label{inter9}
&&\int_0^t\left\Vert U_h(t,s)A_h(s)D_su^h(s)\right\Vert ds\nonumber\\
&\leq& \int_0^t\left\Vert U_h(t,s)(-A_h(s))^{1-\beta/2-\epsilon}\right\Vert_{L(H)}\left\Vert (-A_h(s))^{\beta/2-\epsilon}D_su^h(s)\right\Vert ds\nonumber\\
&\leq& C\int_0^t(t-s)^{-1+\beta/2-\epsilon}s^{\beta/2-\epsilon}ds\leq Ct^{-1+\beta/2-\epsilon}.
\end{eqnarray}
Using \lemref{lemderiv} and \lemref{regularitylemma} yields
\begin{eqnarray}
\label{inter10}
&&\int_0^t\left\Vert U_h(t,s)A_h''(s)u^h(s)\right\Vert ds\nonumber\\
&\leq&\int_0^t\left\Vert U_h(t,s)(-A_h(0))^{1-\beta/2+\epsilon}\right\Vert_{L(H)}\nonumber\\
&&\times\left\Vert (-A_h(0))^{-1+\beta-\epsilon}A_h''(s)(-A_h(0))^{-\beta/2+\epsilon}(-A_h(0))^{\beta/2-\epsilon}u^h(s)\right\Vert ds\nonumber\\
&\leq&C\int_0^t(t-s)^{-1+\beta/2-\epsilon}\left\Vert (-A_h(0))^{\beta/2-\epsilon}u^h(s)\right\Vert ds\nonumber\\
&\leq& C\int_0^t(t-s)^{-1+\beta/2-\epsilon}ds\nonumber\\
&\leq& Ct^{\beta/2-\epsilon}.
\end{eqnarray}
Using (i) yields 
\begin{eqnarray}
\label{inter11}
\int_0^ts\Vert D_su^h(s)\Vert^2ds\leq C\int_0^ts^{-1+\beta}ds\leq Ct^{\beta}.
\end{eqnarray}
Substituting \eqref{inter11}, \eqref{inter10}, \eqref{inter9} and \eqref{inter8} in \eqref{deri8} yields
\begin{eqnarray}
\Vert w^h(t)\Vert\leq Ct^{-1+\beta/2}.
\end{eqnarray}
This completes the proof of the lemma.
\end{proof}

For non commutative operator $H_j$ on Banach space, we define the following product
\begin{eqnarray}
\prod_{j=k}^mH_j=\left\{\begin{array}{ll}
H_mH_{m-1}\cdots H_k\quad\quad \text{if}\quad m\geq k,\\
\mathbf{I}\quad \hspace{3cm}\text{if}\quad m<k.
\end{array}
\right.
\end{eqnarray}
The following stability result is fundamental in our convergence analysis.
\begin{lemma}
\label{stabilite}
Let \assref{assumption1}, \assref{assumption2}, \assref{assumption3} and \assref{assumption4} be fulfilled. Then the following stability estimate holds
{\small
\begin{eqnarray}
\left\Vert \left(\prod_{j=k}^m e^{\left(A_{h,j}+J^h_j\right)\Delta t}\right)(-A_{h,k})^{\gamma}\right\Vert_{L(H)}&\leq& Ct_{m-k+1}^{-\gamma},\quad 0\leq k\leq m\leq M,
\end{eqnarray}
}
for any $\gamma\in[0,1)$.
\end{lemma}
\begin{proof}
As in  \cite[Theorem 1]{Ostermann1}, the main idea is to compare the composition of the perturbed operator with the frozen operator
\begin{eqnarray}
\label{frozen1}
\prod_{j=k}^me^{\left(A_{h,k}+J^h_k\right)\Delta t}=e^{(t_{m+1}-t_k)\left(A_{h,k}+J^h_k\right)}.
\end{eqnarray}
Using \cite[Lemma 9]{Antjd1} yields the following estimate
\begin{eqnarray}
\label{frozen2}
\left\Vert \prod_{j=k}^me^{\left(A_{h,k}+J^h_k\right)\Delta t}(-A_{h,k})^{\gamma}\right\Vert_{L(H)}&=&\left\Vert e^{\left(A_{h,k}+J^h_k\right)t_{m-k+1}}(-A_{h,k})^{\gamma}\right\Vert_{L(H)}\nonumber\\
&\leq& Ct_{m-k+1}^{-\gamma}.
\end{eqnarray}
It remains to bound $\Delta^m_k(-A_{h,k})^{\gamma}$, where $\Delta^m_k$ is defined as follows
\begin{eqnarray}
\label{frozen3}
\Delta^m_k:=\prod_{j=k}^me^{\left(A_{h,j}+J^h_j\right)\Delta t}-\prod_{j=k}^me^{\left(A_{h,k}+J^h_k\right)\Delta t}.
\end{eqnarray}
Using the telescopic identity we obtain
{\small
\begin{eqnarray}
\label{frozen4}
&&\Delta^m_k\nonumber\\
&=&\sum_{j=k+1}^{m-1}\Delta ^m_{j+1}\left(e^{\left(A_{h,j}+J^h_j\right)\Delta t}-e^{\left(A_{h,k}+J^h_k\right)\Delta t}\right)e^{(t_j-t_k)\left(A_{h,k}+J^h_k\right)}\nonumber\\
&+&\sum_{j=k+1}^me^{(t_{m+1}-t_{j+1})\left(A_{h,k}+J^h_k\right)}\left(e^{\left(A_{h,j}+J^h_j\right)\Delta t}-e^{\left(A_{h,k}+J^h_k\right)\Delta t}\right)e^{(t_j-t_k)\left(A_{h,k}+J^h_k\right)}.\nonumber\\
\end{eqnarray}
}
Using the variation of parameter formula \cite[Chapter III, Corollary 1.7]{Engel} yields
\begin{eqnarray}
\label{frozen5}
e^{\left(A_{h,l}+J^h_l\right)\Delta t}=e^{A_{h,l}\Delta t}+\int_0^{\Delta t}e^{A_{h,l}(\Delta t-s)}J^h_le^{\left(A_{h,l}+J^h_l\right)s}ds.
\end{eqnarray}
It follows therefore from \eqref{frozen5} that
\begin{eqnarray}
\label{frozen6}
\left(e^{\left(A_{h,j}+J^h_j\right)\Delta t}-e^{\left(A_{h,k}+J^h_k\right)\Delta t}\right)
&=&\left(e^{A_{h,j}\Delta t}-e^{A_{h,k}\Delta t}\right)\nonumber\\
&+&\int_0^{\Delta t}e^{A_{h,j}(\Delta t-s)}J^h_je^{\left(A_{h,j}+J^h_j\right)s}ds\nonumber\\
&-&\int_0^{\Delta t}e^{A_{h,k}(\Delta t-s)}J^h_ke^{\left(A_{h,k}+J^h_k\right)s}ds\nonumber\\
&=:&IV_1+IV_2+IV_3.
\end{eqnarray}
Using the integral formula of Cauchy exactly as in  \cite[Lemma 1]{Ostermann1} yields
\begin{eqnarray}
\label{frozen7}
\Vert IV_1\Vert_{L(H)}=\left\Vert \left(e^{A_{h,j}\Delta t}-e^{A_{h,k}\Delta t}\right)\right\Vert_{L(H)}\leq C\Delta t.
\end{eqnarray}
Using \cite[Lemma 9]{Antjd1}, \assref{assumption2} and \assref{assumption3} yields
{\small
\begin{eqnarray}
\label{frozen8}
\Vert IV_2\Vert_{L(H)}+\Vert IV_3\Vert_{L(H)} &\leq&2\int_0^{\Delta t}\left\Vert e^{A_{h,k}(\Delta t-s)}\right\Vert_{L(H)}\Vert J^h_k\Vert_{L(H)}\left\Vert e^{\left(A_{h,k}+J^h_k\right)s}\right\Vert_{L(H)}ds\nonumber\\
&\leq& C\int_0^{\Delta t}ds\leq C\Delta t.
\end{eqnarray}
}
Substituting \eqref{frozen8} and \eqref{frozen7} in \eqref{frozen6} yields
\begin{eqnarray}
\label{frozen9}
\left\Vert \left(e^{\left(A_{h,j}+J^h_j\right)\Delta t}-e^{\left(A_{h,k}+J^h_k\right)\Delta t}\right)\right\Vert_{L(H)}\leq C\Delta t.
\end{eqnarray}
Inserting an appropriate power of $(-A_{h,k})^{\gamma}$ in \eqref{frozen4}, using triangle inequality and \eqref{frozen9} yields
\begin{eqnarray}
\label{frozen9}
&&\left\Vert \Delta ^m_k(-A_{h,k})^{\gamma}\right\Vert_{L(H)}\nonumber\\
&\leq&\sum_{j=k+1}^{m-1}\Vert \Delta ^m_{j+1}(-A_{h,k})^{\gamma}\Vert_{L(H)}\Vert (-A_{h,k})^{-\gamma}\Vert_{L(H)}\nonumber\\
&&\times\left\Vert \left(e^{\left(A_{h,j}+J^h_j\right)\Delta t}-e^{\left(A_{h,k}+J^h_k\right)\Delta t}\right)\right\Vert_{L(H)}\left\Vert e^{(t_j-t_k)\left(A_{h,k}+J^h_k\right)}(-A_{h,k})^{\gamma}\right\Vert_{L(H)}\nonumber\\
&+&\sum_{j=k+1}^m\left\Vert e^{(t_{m+1}-t_{j+1})\left(A_{h,k}+J^h_k\right)}\right\Vert_{L(H)}\left\Vert \left(e^{\left(A_{h,j}+J^h_j\right)\Delta t}-e^{\left(A_{h,k}+J^h_k\right)\Delta t}\right)\right\Vert_{L(H)}\nonumber\\
&&\times \left\Vert e^{(t_j-t_k)\left(A_{h,k}+J^h_k\right)}(-A_{h,k})^{\gamma}\right\Vert_{L(H)}\nonumber\\
&\leq& C\Delta t\sum_{j=k+1}^{m-1}\Vert \Delta^m_{j+1}(-A_{h,k})^{\gamma}\Vert_{L(H)} t_{j-k}^{-\gamma}+C\Delta t\sum_{j=k+1}^mt_{j-k}^{-\gamma}\nonumber\\
&\leq& C+C\Delta t\sum_{j=k+1}^{m-1}t_{j-k}^{-\gamma}\Vert \Delta^m_{j+1}(-A_{h,k})^{\gamma}\Vert_{L(H)}.
\end{eqnarray}
Applying the discrete Gronwall's lemma to \eqref{frozen9} yields
\begin{eqnarray}
\label{frozen10}
\left\Vert \Delta ^m_k(-A_{h,k})^{\gamma}\right\Vert_{L(H)}\leq C.
\end{eqnarray}
Using \eqref{frozen10} and \eqref{frozen2} completes the proof of \lemref{stabilite}.
\end{proof}

\begin{lemma}
\label{stabilitysolution}
Let Assumptions \ref{assumption2}, \ref{assumption1} and \ref{assumption3} be fulfilled. Then the numerical scheme \eqref{semi2} satisfies the following estimate
 \begin{eqnarray}
 \Vert u^h_m\Vert\leq R,\quad m\in\{0, 1, \cdots, M\},
 \end{eqnarray}
 where $R>0$ is independent of $h$, $m$, $M$ and $\Delta t$.
\end{lemma}

\begin{proof}
Iterating the numerical solution \eqref{semi2} by substituting $u^h_j$, $j=m-1, \cdots, 1$ only in the first term of \eqref{semi2} by their expressions yields
{\small
\begin{eqnarray}
\label{dim1}
&&u^h_m=\left(\prod_{j=0}^{m-1}e^{\Delta t(A_{h,j}+J^h_j)}\right)u^h_0\\
&+&\sum_{k=0}^{m-1}\int_0^{\Delta t}\left(\prod_{j=m-k}^{m-1}e^{\Delta t(A_{h,j}+J^h_j)}\right)e^{(A_{h, m-k-1}+J^h_{m-k-1})(\Delta t-s)}a^h_{m-k-1}\left(t_{m-k-1}+\frac{\Delta t}{2}\right)ds\nonumber\\
&+&\sum_{k=0}^{m-1}\int_0^{\Delta t}\left(\prod_{j=m-k}^{m-1}e^{\Delta t(A_{h,j}+J^h_j)}\right)e^{(A_{h,m-k-1}+J^h_{m-k-1})(\Delta t-s)}\nonumber\\
&&G^h_{m-k-1}\left(t_{m-k-1}+\frac{\Delta t}{2}, u^h_{m-k-1}\right)ds.\nonumber
\end{eqnarray}
}
Taking the norm in both sides of \eqref{dim1}, using triangle inequality, \lemref{stabilite} and \assref{assumption3} yields
{\small
\begin{eqnarray}
\label{dim2}
\Vert u^h_m\Vert&\leq&\left\Vert\left(\prod_{j=0}^{m-1}e^{\Delta t(A_{h,j}+J^h_j)}\right)\right\Vert_{L(H)}\Vert u^h_0\Vert\\
&+&\sum_{k=0}^{m-1}\int_0^{\Delta t}\left\Vert\left(\prod_{j=m-k}^{m-1}e^{\Delta t(A_{h,j}+J^h_j)}\right)\right\Vert_{L(H)}\left\Vert e^{(A_{h, m-k-1}+J^h_{m-k-1})(\Delta t-s)}\right\Vert_{L(H)}\nonumber\\
&\times&\left\Vert a^h_{m-k-1}\right\Vert\left(t_{m-k-1}+\frac{\Delta t}{2}\right)ds\nonumber\\
&+&\sum_{k=0}^{m-1}\int_0^{\Delta t}\left\Vert\left(\prod_{j=m-k}^{m-1}e^{\Delta t(A_{h,j}+J^h_j)}\right)\right\Vert_{L(H)}\left\Vert e^{\left(A_{h,m-k-1}+J^h_{m-k-1}\right)(\Delta t-s)}\right\Vert_{L(H)}\nonumber\\
&\times&\left\Vert G^h_{m-k-1}\right\Vert_{L(H)}\left\Vert\left(t_{m-k-1}+\frac{\Delta t}{2}, u^h_{m-k-1}\right)\right\Vert ds\nonumber\\
&\leq& C\Vert u^h_0\Vert+C\sum_{k=0}^{m-1}\int_0^{\Delta t}\left(t_{m-k-1}+\frac{\Delta t}{2}\right)ds\nonumber\\
&+&C\sum_{k=0}^{m-1}\int_0^{\Delta t}\left[\left(t_{m-k-1}+\frac{\Delta t}{2}\right)+u^h_{m-k-1}\right]ds.
\end{eqnarray}
}
Using the fact that $t_{m-k-1}+\frac{\Delta t}{2}\leq T$ and $\Vert u^h_0\Vert\leq \Vert u_0\Vert$, it holds from \eqref{dim2} that
\begin{eqnarray}
\label{dim3}
\Vert u^h_m\Vert\leq C\Vert u_0\Vert +C+C\Delta t\sum_{k=0}^{m-1}\Vert u^h_k\Vert.
\end{eqnarray}
Applying the discrete Gronwall"s lemma to \eqref{dim3} yields
\begin{eqnarray}
\label{dim4}
\Vert u^h_m\Vert\leq C(1+\Vert u_0\Vert)\leq R,\quad m\in\{0, \cdots, M\}.
\end{eqnarray}
This completes the proof of \lemref{stabilitysolution}.
\end{proof}

\begin{lemma}
\label{lemmestrategie1}
Let Assumptions \ref{assumption2} and \ref{assumption3} be fulfilled. Then the fractional powers of $-(A_{h,k}+J^h_k)$ exist and the following estimate holds
\begin{eqnarray}
\Vert \left(-(A_{h, k}+J^h_k)\right)^{-\alpha}\Vert_{L(H)}\leq C,\quad \alpha>0,
\end{eqnarray}
with $C$ independent of $h$ and $k$.
\end{lemma}
\begin{proof}
First of all we claim that $e^{(A_{h,k}+J^h_k)t}$ is uniformly exponentially stable. In fact,  from the variation of parameters formula \cite[Chapter 3, Corollary 1.7]{Engel} or \cite[Page 77, Section 3.1]{Pazy} it holds that
 \begin{eqnarray}
 \label{strategie2}
 e^{(A_{h,k}+J^h_k)t}=e^{A_{h,k}t}+\int_0^te^{A_{h,k}(t-s)}J^h_ke^{(A_{h,k}+J^h_k)s}ds,\quad t\geq 0.
 \end{eqnarray}
 Taking the norm in both sides of \eqref{strategie2}, inserting  appropriately power of $(-A_{h,k})^{-\gamma}(-A_{h,k})^{\gamma}$ (with $\gamma\in(0, 1)$), using the uniformly boundedness of $(-A_{h,k})^{-\gamma}$, \assref{assumption3} and \eqref{smooth2}  yields
 \begin{eqnarray}
 \label{strategie3}
 &&\Vert e^{(A_{h,k}+J^h_k)t}\Vert_{L(H)}\nonumber\\
 &\leq& \Vert (-A_{h,k})^{-\gamma}\Vert_{L(H)}\Vert (-A_{h,k})^{\gamma}e^{A_{h,k}t}\Vert_{L(H)}\nonumber\\
 &+&\int_0^t\Vert (-A_{h,k})^{-\gamma}\Vert_{L(H)}\Vert(-A_{h,k})^{\gamma} e^{A_{h,k}(t-s)}\Vert_{L(H)}\Vert J^h_k\Vert_{L(H)}\Vert e^{(A_{h,k}+J^h_k)s}\Vert_{L(H)}ds\nonumber\\
 &\leq& Ct^{-\gamma}+C\int_0^t(t-s)^{-\gamma}\Vert e^{(A_{h,k}+J^h_k)s}\Vert_{L(H)}ds.
 \end{eqnarray}
 Applying the generalized Gronwall's lemma \cite[Lemma 3.5.2]{Henry} to \eqref{strategie3} yields
 \begin{eqnarray}
 \label{strategie4}
 \Vert e^{(A_{h,k}+J^h_k)t}\Vert_{L(H)}\leq Ct^{-\gamma}=\frac{C}{t^{\gamma}},\quad \gamma\in(0,1),\quad t\geq 0.
\end{eqnarray}  
 Taking the limit as $t$ goes to $\infty$ in \eqref{strategie4} yields
 \begin{eqnarray}
 \lim_{t\longrightarrow\infty} \Vert e^{(A_{h,k}+J^h_k)t}\Vert_{L(H)}=0.
 \end{eqnarray}
 Employing \cite[Proposition 1.7, Chapter V, Page 299]{Engel}, it follows that $e^{(A_{h, k}+J^h_k)t}$ is  exponentially stable, i.e. there exists two positive constants $L_k$ and $\omega_k$ such that
 \begin{eqnarray}
 \label{strategie5}
 \Vert e^{(A_{h,k}+J^h_k)t}\Vert_{L(H)}\leq L_ke^{-\omega_kt},\quad t\geq 0.
 \end{eqnarray}
  Let $B[0, R]:=\{v\in H: \Vert v\Vert\leq R\}$, where $R$ is defined in \lemref{stabilitysolution}. 
 More generally, for every $\tau\in[0, T]$ and $v\in B[0, R]$ there two positive constants $L_{\tau, v}$ and $\omega_{\tau,v}$ such that
 \begin{eqnarray}
 \label{strategie7a}
 \Vert e^{(A_{h}(\tau)+J^h_{\tau,v})t}\Vert_{L(H)}\leq L_{\tau,v}e^{-\omega_{\tau,v}t},\quad t\geq 0,
 \end{eqnarray}
 where $J^h_{\tau, v}:=P_h\frac{\partial F}{\partial v}(\tau, v)$. 
 Note that the function  $(\tau, v)\longmapsto \omega_{\tau, v}$ is continuous. This  follows from the definition of the growth bound $\omega_{\tau, v}$
 \begin{eqnarray}
 \omega_{\tau, v}:=\inf_{t>0}\frac{1}{t}\log\left\Vert e^{(A_h(\tau)+J^h_{\tau, v})t}\right\Vert_{L(H)},\quad \tau\in[0, T],\quad v\in B[0, R].
\end{eqnarray}   
Due to \eqref{strategie7a}, the following constant is well defined
\begin{eqnarray}
\label{strategie7b}
L'_{\tau, v}:=\sup_{t\geq 0}\left\Vert e^{\left(A_h(\tau)+J^h_{\tau, v}\right)t}\right\Vert_{L(H)}e^{\omega_{\tau, v}t},\quad \tau\in[0, T],\quad v\in B[0, R].
\end{eqnarray}
It follows from the above definition \eqref{strategie7b} that the function $(\tau, v)\longmapsto L'_{\tau, v}$  is continuous. 
Therefore by Weierstrass's theorem there exist two positive constants $L'$ and $\omega$ such that
 \begin{eqnarray}
 L'=\sup_{\tau\in[0, T], v\in B(0, R)}L'_{\tau, v},\quad \omega=\inf_{\tau\in[0, T], v\in B(0, R)}\omega_{\tau, v}.
 \end{eqnarray}
 Consequently, we have
 \begin{eqnarray}
 \label{strategie6}
 \Vert e^{(A_{h,k}+J^h_k)t}\Vert_{L(H)}\leq L'e^{-\omega t},\quad t\geq 0,\quad k\in\{0,1,\cdots,M\}.
 \end{eqnarray}
 This proves the claim. Let us now finish the proof of \lemref{lemmestrategie1}.
  Assumptions \ref{assumption2} and \ref{assumption3} imply that $-(A_{h, k}+J^h_k)$ is a   positive operator. Therefore its fractional powers are well defined and are given by 
 \begin{eqnarray}
 \label{strategie1}
 \left(-(A_{h, k}+J^h_k)\right)^{-\alpha}=\frac{1}{\Gamma(\alpha)}\int_0^{\infty}t^{\alpha-1}e^{(A_{h,k}+J^h_k)t}dt, 
 \end{eqnarray}
 where $\Gamma(\alpha)$ is a gamma function, see e.g. \cite{Henry,Pazy,Engel}.
 Taking the norm in both sides of \eqref{strategie1} and using \eqref{strategie6} yields
 \begin{eqnarray}
 \left\Vert\left(-(A_{h, k}+J^h_k)\right)^{-\alpha}\right\Vert_{L(H)}&\leq&\frac{L'}{\Gamma(\alpha)}\int_0^{\infty}t^{\alpha-1}e^{-\omega t}dt=\frac{L'\omega^{2-\alpha}}{\Gamma(\alpha)}\int_0^{\infty}s^{\alpha-1}e^{-s}ds\nonumber\\
 &=&L'\omega^{2-\alpha}<\infty.
 \end{eqnarray}
 This completes the proof of the lemma.
\end{proof}

\begin{lemma}
\label{lemmestrategie2}
Let Assumptions \ref{assumption2} and \ref{assumption3} be fulfilled. Then the following estimate holds
\begin{eqnarray}
\label{strategie8a}
\left\Vert \left(-(A_{h,k}+J^h_k)\right)^{-\alpha}(-A_{h,k})^{\alpha}\right\Vert_{L(H)}\leq C,\quad \alpha\in[0, 1]\\
\label{strategie8b}
\left\Vert (-A_{h,k})^{\alpha}\left(-(A_{h,k}+J^h_k)\right)^{-\alpha}\right\Vert_{L(H)}\leq C,\quad \alpha\in[0, 1].
\end{eqnarray}
\end{lemma}
\begin{proof}
We only prove \eqref{strategie8a} since the proof of \eqref{strategie8b} is similar.
For $\alpha=1$, using triangle inequality, \assref{assumption3} and \lemref{lemmestrategie1} it holds that
{\small
\begin{eqnarray}
\label{strategie8}
&&\left\Vert \left(-(A_{h,k}+J^h_k)\right)^{-1}(-A_{h,k})\right\Vert_{L(H)}\nonumber\\
&\leq& \left\Vert \left(-(A_{h,k}+J^h_k)\right)^{-1}\left(-(A_{h,k}+J^h_k)\right)\right\Vert_{L(H)}+\left\Vert \left(-(A_{h,k}+J^h_k)\right)^{-1}\right\Vert_{L(H)}\Vert J^h_k\Vert_{L(H)}\nonumber\\
&\leq& C.
\end{eqnarray}
}
Note that \eqref{strategie8a} obviously holds for $\alpha=0$. As in \cite{Antjd1,Antjd2,Thomee2} the intermediates cases follow by interpolation technique.
\end{proof}

\begin{lemma}
\label{utilemma}
For $k=0,\cdots, M-1$ and $t_k\leq t\leq t_{k+1}$, let us set
\begin{eqnarray}
\label{reste1}
L^h_k(t)&:=&\left(A_h(t)-A_{h,k}\right)u^h(t)-a^h_k(t_k+t)\nonumber\\
&+&G^h_k\left(t,u^h(t)\right)-G^h_k\left(t_k+\frac{\Delta t}{2},u^h(t_k)\right).
\end{eqnarray}
Under \assref{assumption1}, \assref{assumption2}, \assref{assumption3} and \assref{assumption4},  provided that $L^h_k$ is twice differentiable on $(t_k, t_{k+1})$, the following estimates hold
\begin{eqnarray}
\label{interessant1}
\left\Vert (-A_h(0))^{-\epsilon}\left(L^h_k\right)'\left(t_k+\frac{\Delta t}{2}\right)\right\Vert&\leq& Ct_k^{-1+\beta/2},\quad k\geq 1,\\
\label{interessant2}
\left\Vert \left(-(A_{h,k}+J^h_k)\right)^{-\epsilon}\left(L^h_k\right)'\left(t_k+\frac{\Delta t}{2}\right)\right\Vert&\leq& Ct_k^{-1+\beta/2},\quad k\geq 1,\\
\label{interessant3}
\left\Vert \left(-A_h(0)\right)^{-1}\left(L^h_k\right)''\left(t\right)\right\Vert&\leq& Ct^{-2+\beta/2},\quad t>0,\\
\label{interessant4}
\left\Vert \left(-(A_{h,k}+J^h_k)\right)^{-\epsilon}\left(L^h_k\right)''\left(t\right)\right\Vert&\leq& Ct^{-2+\beta/2},\quad t>0,
\end{eqnarray}
where $\epsilon>0$ is a positive number, small enough. 
\end{lemma}
\begin{proof}
Let us start with the estimate of \eqref{interessant1}. 
Taking the derivative in both sides of \eqref{reste1}, using  \eqref{remainder2}  and \eqref{remainder1} yields
\begin{eqnarray}
\label{vend1}
(L^h_k)'(t)&=& A_h'(t)u^h(t)+\left(A_h(t)-A_{h,k}\right)D_tu^h(t)+P_h\frac{\partial F}{\partial t}\left(t,u^h(t)\right)\nonumber\\
&+&P_h\frac{\partial F}{\partial u}\left(t,u^h(t)\right)D_tu^h(t)-P_h\frac{\partial F}{\partial u}\left(t_k+\frac{\Delta t}{2},u^h_k\right)D_tu^h(t)\nonumber\\
&-&P_h\frac{\partial F}{\partial t}\left(t_k+\frac{\Delta t}{2}, u^h_k\right)-a^h_k.
\end{eqnarray}
Taking the norm in both sides of \eqref{vend1}, using \lemref{utilemma},  \assref{assumption3}, \lemref{regularitylemma}, \lemref{lemderiv}, \lemref{ancien}
and the fact that $(-A_h(0))^{-\epsilon}$ is bounded yields
{\small
\begin{eqnarray}
\label{vend2}
&&\left\Vert (-A_h(0))^{-\epsilon}\left(L^h_k\right)'\left(t_k+\frac{\Delta t}{2}\right)\right\Vert\nonumber\\
&\leq& \left\Vert (-A_h(0))^{-\epsilon} A_h'\left(t_k+\frac{\Delta t}{2}\right)u^h\left(t_k+\frac{\Delta t}{2}\right)\right\Vert+C\left\Vert P_h\frac{\partial F}{\partial t}\left(t_k+\frac{\Delta t}{2}, u^h\left(t_k+\frac{\Delta t}{2}\right)\right)\right\Vert\nonumber\\
&+&C\left\Vert P_h\frac{\partial F}{\partial u}\left(t_k+\frac{\Delta t}{2}, u^h\left(t_k+\frac{\Delta t}{2}\right)\right)\right\Vert_{L(H)}\left\Vert D_tu^h\left(t_k+\frac{\Delta t}{2}\right)\right\Vert+C\left\Vert P_h\frac{\partial F}{\partial t}\left(t_k+\frac{\Delta t}{2}, u^h_k\right)\right\Vert\nonumber\\
&+&C\left\Vert P_h\frac{\partial F}{\partial u}\left(t_k, u^h_k\right)\right\Vert_{L(H)}\left\Vert D_tu^h\left(t_k+\frac{\Delta t}{2}\right)\right\Vert+C\left\Vert \frac{\partial F}{\partial t}\left(t_k+\frac{\Delta t}{2}, u^h_k\right)\right\Vert\nonumber\\
&\leq& \left\Vert (-A_h(0))^{-\epsilon}A_h'\left(t_k+\frac{\Delta t}{2}\right)(-A_h(0))^{-1+\epsilon}\right\Vert_{L(H)}\left\Vert (-A_h(0))^{1-\epsilon}u^h\left(t_k+\frac{\Delta t}{2}\right)\right\Vert\nonumber\\
&+&C+C\left\Vert u^h\left(t_k+\frac{\Delta t}{2}\right)\right\Vert+C\left\Vert D_tu^h\left(t_k+\frac{\Delta t}{2}\right)\right\Vert+C\nonumber\\
&\leq&C\left(t_k+\frac{\Delta t}{2}\right)^{-1+\epsilon+\beta/2}+C\left(t_k+\frac{\Delta t}{2}\right)^{-1+\beta/2}\leq Ct_k^{-1+\beta/2}.
\end{eqnarray}
}
This completes the proof of \eqref{interessant1}. 

Let us now prove \eqref{interessant2}. Inserting an appropriate power of $-A_{h, k}$, using \eqref{interessant1}, Lemmas \ref{lemma0} and \ref{lemmestrategie2} yields
\begin{eqnarray}
&&\left\Vert \left(-(A_{h,k}+J^h_k)\right)^{-\epsilon}\left(L^h_k\right)'\left(t_k+\frac{\Delta t}{2}\right)\right\Vert\nonumber\\
&\leq& \left\Vert \left(-(A_{h,k}+J^h_k)\right)^{-\epsilon}(-A_{h,k})^{\epsilon}\right\Vert_{L(H)}\left\Vert\left(-A_{h,k}\right)^{-\epsilon}\left(L^h_k\right)'\left(t_k+\frac{\Delta t}{2}\right)\right\Vert\nonumber\\
&\leq& Ct_k^{-1+\beta/2}.
\end{eqnarray}
This completes the proof of \eqref{interessant2}. Let us complete the proof of the lemma with \eqref{interessant3}.
Taking the derivative in both sides of \eqref{vend1} yields
{\small
\begin{eqnarray}
\label{vend3}
\left(L^h_k\right)''(t)&=&A_h''(t)u^h(t)+2A_h'(t)D_tu^h(t)+A_h(t)D_t^2u^h(t)
+P_h\frac{\partial^2F}{\partial t^2}\left(t,u^h(t)\right)\nonumber\\
&+&2P_h\frac{\partial^2F}{\partial t\partial u}\left(t, u^h(t)\right)D_tu^h(t)+P_h\frac{\partial^2F}{\partial u^2}\left(t,u^h(t)\right)\left(D_tu^h(t), D_tu^h(t)\right).
\end{eqnarray}
}
Inserting $(-A_h(0))^{-1}$ in \eqref{vend3}, taking the norm in both sides,  using \lemref{lemderiv}, \lemref{regularitylemma}, \lemref{ancien} and the fact that $(-A_h(0))^{-1}$ is bounded yields
\begin{eqnarray}
\label{vend4}
&&\Vert \left(-A_h(0))^{-1}(L^h_k\right)''(t)\Vert\nonumber\\
&\leq&\Vert (-A_h(0))^{-1}A_h''(t)\Vert_{L(H)}\Vert u^h(t)\Vert+2\Vert (-A_h(0))^{-1}A_h'(t)\Vert_{L(H)}\Vert D_tu^h(t)\Vert\nonumber\\
&+&\Vert (-A_h(0))^{-1}A_h(t)\Vert_{L(H)}\Vert D^2_tu^h(t)\Vert+C\left\Vert\frac{\partial^2F}{\partial t^2}\left(t,u^h(t)\right)\right\Vert\nonumber\\
&+&C\left\Vert\frac{\partial^2F}{\partial t\partial u}\left(t,u^h(t)\right)\right\Vert_{L(H)}\Vert D_tu^h(t)\Vert+C\left\Vert\frac{\partial^2F}{\partial t\partial u}\left(t,u^h(t)\right)\right\Vert_{L(H)}\left\Vert D_tu^h(t)\right\Vert\nonumber\\
&+&C\left\Vert\frac{\partial^2F}{\partial u^2}\left(t,u^h(t)\right)\right\Vert_{L(H\times H, H)}\left\Vert D_t^2u^h(t)\right\Vert\nonumber\\
&\leq& C+Ct^{-1+\beta/2}+Ct^{-2+\beta/2}\leq Ct^{-2+\beta/2}.
\end{eqnarray}
The proof of \eqref{interessant4} is similar to that of \eqref{interessant2}.
This completes the proof of \lemref{utilemma}.
\end{proof}

\begin{lemma}
\label{bonlemma}
Let \assref{assumption2} be fulfilled,  let $m\in\{0, 1, \cdots, M\}$ and $0<t\leq T$. Then the following estimate holds
\begin{eqnarray}
\label{utile1}
\left\Vert \left(-(A_{h,m}+J^h_m)\right)^{\alpha}e^{\left(A_{h,m}+J^h_m\right)t}\right\Vert_{L(H)}&=&
\left\Vert e^{\left(A_{h,m}+J^h_m\right)t}\left(-(A_{h,m}+J^h_m)\right)^{\alpha}\right\Vert_{L(H)}\nonumber\\
&\leq& Ct^{-\alpha},\quad \alpha\in[0, 1].
\end{eqnarray}
Moreover, for
 $0\leq \alpha_1\leq\alpha_2\leq 1$  and  any $0\leq t\leq T$, the following estimate holds
{\small
\begin{eqnarray}
\label{utile2}
\left\Vert (-(A_{h,m}+J^h_m))^{-\alpha_1}\varphi_1(\Delta t(A_{h,m}+J^h_m))(-(A_{h,m}+J^h_m))^{\alpha_2}
\right\Vert_{L(H)}\leq C\Delta t^{\alpha_1-\alpha_2}.
\end{eqnarray}
}
\end{lemma}
\begin{proof}
Let us start with the proof of \eqref{utile1}. Note that for $\alpha=1$, using Assumption \ref{assumption2} and \ref{assumption3}, we have 
\begin{eqnarray}
\label{utile1a}
\Vert e^{A_{h,m}t}(-(A_{h,m}+J^h_m))\Vert_{L(H)}&\leq& \Vert e^{A_{h,m}t}A_{h,m}\Vert_{L(H)}+\Vert e^{A_{h,m}t}J^h_m\Vert_{L(H)}\nonumber\\
&\leq& Ct^{-1}+C\leq Ct^{-1}.
\end{eqnarray}
From \eqref{strategie2}, it holds that
\begin{eqnarray}
\label{utile1b}
e^{(A_{h,m}+J^h_m)t}(-(A_{h,m}+J^h_m))&=&e^{A_{h,m}t}(-(A_{h,m}+J^h_m))\\
&+&\int_0^te^{A_{h,m}(t-s)}J^h_me^{(A_{h,m}+J^h_m)s}(-(A_{h,m}+J^h_m))ds.\nonumber
\end{eqnarray}
Taking the norm in both sides of \eqref{utile1b} and using \eqref{utile1a} yields
\begin{eqnarray}
\label{utile1c}\left\Vert e^{(A_{h,m}+J^h_m)t}(-(A_{h,m}+J^h_m))\right\Vert_{L(H)}&\leq& Ct^{-1}\\
&+&C\int_0^t\left\Vert e^{(A_{h,m}+J^h_m)s}(-(A_{h,m}+J^h_m))\right\Vert_{L(H)}ds.\nonumber
\end{eqnarray}
Applying the Gronwall's lemma to \eqref{utile1c} yields
\begin{eqnarray}
\label{utile1d}
\left\Vert e^{(A_{h,m}+J^h_m)t}(-(A_{h,m}+J^h_m))\right\Vert_{L(H)}&\leq& Ct^{-1}.
\end{eqnarray}
Note that \eqref{utile1} obviously holds for $\alpha=0$. The intermediate cases therefore follow by interpolation technique and the proof of \eqref{utile1} is completes.
Let us now prove \eqref{utile2}. 
 From \eqref{phi1}, it holds that
\begin{eqnarray}
\label{utile4}
&&(-(A_{h,m}+J^h_m))^{-\alpha_1}\varphi_1\left(\Delta t\left(A_{h,m}+J^h_m\right)\right)(-(A_{h,m}+J^h_m))^{\alpha_2}\nonumber\\
&=&\frac{1}{\Delta t}\int_0^{\Delta t}e^{\left(A_{h,m}+J^h_m\right)(\Delta t-s)}(-(A_{h,m}+J^h_m))^{\alpha_2-\alpha_1}ds.
\end{eqnarray}
Taking the norm in both sides of \eqref{utile4} and using \eqref{utile1} yields
{\small
\begin{eqnarray}
\Vert (-A_h(0))^{-\alpha_1}\varphi_1\left(\Delta t\left(A_{h,m}+J^h_m\right)\right)(-A_h(0))^{\alpha_2}\Vert_{L(H)}&\leq& C\Delta t^{-1}\int_0^{\Delta t}(\Delta t-s)^{\alpha_1-\alpha_2}ds\nonumber\\
&\leq& C\Delta t^{\alpha_1-\alpha_2}.
\end{eqnarray}
}
This proves  \eqref{utile2}, and the proof of \lemref{bonlemma} is completed. 
\end{proof}

The following lemma can be found in \cite{Stig2}.
\begin{lemma}
\label{lemmastig}
For all $\alpha_1, \alpha_2>0$ and $\alpha\in[0,1)$, there exist two positive constants $C_{\alpha_1,\alpha_2}$ and $C_{\alpha,\alpha_2}$  such that
\begin{eqnarray}
\label{stig}
\Delta t\sum_{j=1}^mt_{m-j}^{-1+\alpha_1}t_j^{-1+\alpha_2}&\leq& C_{\alpha_1,\alpha_2}t_m^{-1+\alpha_1+\alpha_2},\\
 \Delta t\sum_{j=1}^mt_{m-j}^{-\alpha}t_j^{-1+\alpha_2}&\leq& C_{\alpha,\alpha_2}t_m^{-\alpha+\alpha_2}.
\end{eqnarray}
\end{lemma}
\begin{proof}
The proof of the first estimate of \eqref{stig} follows from the comparison with the following integral
\begin{eqnarray}
\int_0^t(t-s)^{-1+\alpha_1}s^{-1+\alpha_2}ds.
\end{eqnarray}
The proof of the second estimate of \eqref{stig} is a consequence of the first estimate.
\end{proof}

\subsection{Proof of \thmref{mainresult1}}
\label{proof_main}
We split the error term in two parts via triangle inequality as follows
\begin{eqnarray}
\Vert u(t_m)-u^h_m\Vert\leq \Vert u(t_m)-u^h(t_m)\Vert+\Vert u^h(t_m)-u^h_m\Vert=:V_1+V_2.
\end{eqnarray}
The space error $V_1$ is estimated in \propref{proposition2}. It remains to estimate the time error $V_2$. 
The initial value problem \eqref{semi} in the subinterval  $[t_m, t_{m+1}]$ can be written in the following form
{\small
\begin{eqnarray}
\frac{du^h}{dt}&=&\left[A_{h,m}+J^h_m\right]u^h(t)+a^h_mt+G^h_m\left(t_m+\frac{\Delta t}{2},u^h(t_m)\right)\nonumber\\
&+&\left(A_h(t)-A_{h,m}\right)u^h(t)+G^h_m(t,u^h(t))-G^h_m\left(t_m+\frac{\Delta t}{2},u^h(t_m)\right).
\end{eqnarray}
}
Consequently, by the variation of constant formula, we have the following representation of the exact solution
{\small
\begin{eqnarray}
\label{semi3}
&&u^h(t_{m+1})\nonumber\\
&=&e^{\left(A_{h,m}+J^h_m\right)\Delta t}u^h(t_m)+\int_0^{\Delta t}e^{\left(A_{h,m}+J^h_m\right)(\Delta t-s)}L^h_m(s+t_m)ds\nonumber\\
&+&\int_0^{\Delta t}e^{\left(A_{h,m}+J^h_m\right)(\Delta t-s)}\left[G^h_m\left(t_m+\frac{\Delta t}{2},u^h(t_m)\right)+a^h_m(t_m+s)\right]ds
\end{eqnarray}
}
where $L^h_k(t)$ is defined in \lemref{utilemma}. 
Let $e^h_{m+1}:=u^h_{m+1}-u^h(t_{m+1})$ be the time error at  $t_{m+1}$ and $\delta^h_{m+1}$ be the defect defined by
\begin{eqnarray}
\label{defect}
\delta^h_{m+1}:=\int_0^{\Delta t}e^{\left(A_{h,m}+J^h_m\right)(\Delta t-s)}L^h_m(s+t_m)ds.
\end{eqnarray}
Taking the difference between \eqref{semi2} and \eqref{semi3} yields
{\small
\begin{eqnarray}
\label{recursion}
e^h_{m+1}&=&e^{\left(A_{h,m}+J^h_m\right)\Delta t}e^h_m-\delta^h_{m+1}\nonumber\\
&+&\Delta t\varphi
_1\left(\Delta t(A_{h,m}+J^h_m)\right)\left[G^h_m\left(t_m+\frac{\Delta t}{2},u^h_m\right)-G^h_m\left(t_m+\frac{\Delta t}{2},u^h(t_m)\right)\right].
\end{eqnarray}
}
Iterating the error recursion \eqref{recursion} and using the fact that $e^h_0=0$ yields 
{\small
\begin{eqnarray}
e^h_{m}&=&\sum_{k=0}^{m-1}S^h_{m-1,k+1}\left[\Delta t\varphi_1\left(\Delta t(A_{h,k}+J^h_k)\right)\left(G^h_k\left(t_k+\frac{\Delta t}{2},u^h_k\right)-G^h_k\left(t_k+\frac{\Delta t}{2},u^h(t_k)\right)\right)-\delta^h_{k+1}\right]\nonumber\\
&=&\Delta t\sum_{k=0}^{m-1}S^h_{m-1,k+1}\varphi_1\left(\Delta t(A_{h,k}+J^h_k)\right)\left(G^h_k\left(t_k+\frac{\Delta t}{2},u^h_k\right)-G^h_k\left(t_k+\frac{\Delta t}{2},u^h(t_k)\right)\right)\nonumber\\
&-&\sum_{k=0}^{m-1}S_{m-1,k+1}^h\varphi_1\left(\Delta t\left(A_{h,k}+J^h_k\right)\right)\delta^h_{k+1}\nonumber\\
&=:&J_1+J_2,
\end{eqnarray}
}
where
\begin{eqnarray}
S^h_{m, k}:=\left(\prod_{j=k}^me^{\Delta t(A_{h, j}+J^h_j)}\right),\quad m, k\in\mathbb{N}.
\end{eqnarray}
Using triangle inequality, \lemref{stabilite} and \eqref{remainder4} yields 
\begin{eqnarray}
\Vert J_1\Vert\leq C\Delta t\sum_{k=0}^{m-1}\Vert S^h_{m-1,k+1}\Vert_{L(H)}\Vert e^h_k\Vert\leq C\Delta t\sum_{k=0}^{m-1}\Vert e^h_k\Vert.
\end{eqnarray}
We therefore obtain the following estimate
\begin{eqnarray}
\label{cast0}
\Vert e^h_{m}\Vert\leq C\Delta t\sum_{k=0}^{m-1}\Vert e^h_k\Vert+\Vert J_2\Vert.
\end{eqnarray}
Assuming that the map $L^h_k$ is twice differentiable on $(t_k, t_{k+1})$, we obtain the following Taylor expansion
{\small
\begin{eqnarray}
\label{Taylor}
L^h_k(s+t_k)&=&\left(s-\frac{\Delta t}{2}\right)\left(L^h_k\right)'\left(t_k+\frac{\Delta t}{2}\right)\nonumber\\
&+&\left(s-\frac{\Delta t}{2}\right)^2\int_0^1(1-\sigma)\left(L^h_k\right)''\left(t_k+\frac{\Delta t}{2}+\sigma\left(s-\frac{\Delta t}{2}\right)\right)d\sigma,
\end{eqnarray}
}
where $0<s<\Delta t$. Let the linear operator $\varphi_2$ be defined as follows
\begin{eqnarray}
\label{phi2}
\quad\varphi_2\left(\Delta t\left(A_{h,m}+J^h_m\right)\right)&:=&\frac{1}{\Delta t^2}\int_0^{\Delta t}e^{\left(A_{h,m}+J^h_m\right)(\Delta t-s)}sds. 
\end{eqnarray}
The functions $\varphi_1$ and $\varphi_2$ satisfy the following relation
\begin{eqnarray}
\label{relation}
\varphi_2(z)=\frac{\varphi_1(z)-1}{z}.
\end{eqnarray}
Note that the operators $\varphi_1$ and $\varphi_2$ defined respectively in \eqref{phi1} and \eqref{phi2} also satisfy the following relation
\begin{eqnarray}
\label{Oster}
&&\varphi_2\left(\Delta t\left(A_{h,m}+J^h_m\right)\right)-\frac{1}{2}\varphi_1\left(\Delta t\left(A_{h,m}+J^h_m\right)\right)\nonumber\\
&=&\Delta t\left(A_{h,m}+J^h_m\right)\chi\left(\Delta t\left(A_{h,m}+J^h_m\right)\right),
\end{eqnarray}
where $\chi\left(\Delta t\left(A_{h,m}+J^h_m\right)\right)$ is a bounded linear operator. In particular, as in \cite[(20)]{Ostermann1} or \cite[(2.8b)]{Gonza1},
one can easily check by using \cite[Lemma 9]{Antjd1} that the following estimates hold for any $\gamma\geq 0$
{\small
\begin{eqnarray}
\label{ganza1}
\Vert \varphi_1\left(\Delta t\left(A_{h,m}+J^h_m\right)\right)\Vert_{L(H)}+\Vert \varphi_2\left(\Delta t\left(A_{h,m}+J^h_m\right)\right)\Vert_{L(H)}&\leq& C,\\
\label{gonza2}
\Vert \left(-(A_{h,k}+J^h_k)\right)^{-\gamma}\chi\left(\Delta t\left(A_{h,m}+J^h_m\right)\right)\left(-(A_{h,k}+J^h_k)\right)^{\gamma}\Vert_{L(H)}&\leq& C.
\end{eqnarray}
}
Taking in account \eqref{Taylor} and \eqref{Oster}, the defect \eqref{defect} can be written as follows
{\small
\begin{eqnarray}
\label{repre1}
\delta^h_k&=&\Delta t^2\left(\varphi_2\left(\Delta t\left(A_{h,k}+J^h_k\right)\right)-\frac{1}{2}\varphi_1\left(\Delta t\left(A_{h,k}+J^h_k\right)\right)\right)\left(L^h_k\right)'\left(t_k+\frac{\Delta t}{2}\right)\\
&+&\int_0^{\Delta t}e^{(\Delta t-s)\left(A_{h,k}+J^h_k\right)}\left(s-\frac{\Delta t}{2}\right)^2\int_0^1(1-\sigma)\left(L^h_k\right)''\left(t_k+\frac{\Delta t}{2}+\sigma\left(s-\frac{\Delta t}{2}\right)\right)d\sigma ds.\nonumber
\end{eqnarray}
}
Substituting \eqref{Oster} in \eqref{repre1} yields
{\small
\begin{eqnarray}
\label{repre2}
\delta^h_k&=&\Delta t^3\left(A_{h,k}+J^h_k\right)\chi\left(\Delta t\left(A_{h,k}+J^h_k\right)\right)\left(L^h_k\right)'\left(t_k+\frac{\Delta t}{2}\right)\nonumber\\
&+&\int_0^{\Delta t}e^{(\Delta t-s)\left(A_{h,k}+J^h_k\right)}\left(s-\frac{\Delta t}{2}\right)^2\int_0^1(1-\sigma)\left(L^h_k\right)''\left(t_k+\frac{\Delta t}{2}+\sigma\left(s-\frac{\Delta t}{2}\right)\right)d\sigma ds.\nonumber\\
&=:&\delta^{(1)h}_k+\delta^{(2)h}_k.
\end{eqnarray}
}
Before proceeding further, we claim that
\begin{eqnarray}
\label{fact2}
\left\Vert \left(-(A_{h,k}+J^h_k)\right)^{-1}\delta^{(2)h}_k\right\Vert \leq C\Delta t^3t_k^{-2+\beta/2}.
\end{eqnarray}
In fact, using \lemref{utilemma}, \lemref{lemma0} and \cite[Lemma 9]{Antjd1} it holds that
{\small
\begin{eqnarray}
\label{mardi1}
&&\left\Vert \left(-(A_{h,k}+J^h_k)\right)^{-1}\delta^{(2)h}_k\right\Vert\nonumber\\
&\leq& C\int_0^{\Delta t}\left\Vert e^{(\Delta t-s)\left(A_{h,k}+J^h_k\right)}\right\Vert_{L(H)}\left(s-\frac{\Delta t}{2}\right)^2\nonumber\\
&&\int_0^1(1-\sigma)\left\Vert \left(-(A_{h,k}+J^h_k)\right)^{-1} \left(L^h_k\right)''\left(t_k+\frac{\Delta t}{2}+\sigma\left(s-\frac{\Delta t}{2}\right)\right)\right\Vert d\sigma ds\\
&\leq& C\int_0^{\Delta t}\left(s-\frac{\Delta t}{2}\right)^2\int_0^1(1-\sigma)\left\Vert \left(-(A_{h,k}+J^h_k)\right)^{-1} \left(L^h_k\right)''\left(t_k+\frac{\Delta t}{2}+\sigma\left(s-\frac{\Delta t}{2}\right)\right)\right\Vert d\sigma ds.\nonumber
\end{eqnarray}
}
Since 
\begin{eqnarray}
\frac{\Delta t}{2}+\sigma\left(s-\frac{\Delta t}{2}\right)\geq 0,\quad s\in[0,\Delta t],\quad \sigma\in[0,1],
\end{eqnarray}
it follows from \lemref{utilemma} that
\begin{eqnarray}
\label{mardi2}
\left\Vert \left(-(A_{h,k}+J^h_k)\right)^{-1} \left(L^h_k\right)''\left(t_k+\frac{\Delta t}{2}+\sigma\left(s-\frac{\Delta t}{2}\right)\right)\right\Vert \leq Ct_k^{-2+\beta/2}, 
\end{eqnarray}
for $s\in[0,\Delta t]$ and $\quad \sigma\in[0,1]$.
Substituting \eqref{mardi2} in \eqref{mardi1} yields
\begin{eqnarray}
\left\Vert \left(-(A_{h,k}+J^h_k)\right)^{-1} \delta^{(2)h}_k\right\Vert&\leq& C\int_0^{\Delta t}\int_0^1(1-\sigma)\left(s-\frac{\Delta t}{2}\right)^2t_k^{-2+\beta/2}d\sigma ds\nonumber\\
&\leq& C\Delta t^3t_k^{-2+\beta/2}.
\end{eqnarray}
We can also easily check that
\begin{eqnarray}
\label{very1}
\left\Vert \left(-(A_{h,k}+J^h_k)\right)^{-1-\epsilon}\delta_k^{(1)h}\right\Vert\leq C\Delta t^3t_k^{-1+\beta/2}.
\end{eqnarray}
In fact, employing \lemref{utilemma} and \eqref{gonza2}, it holds that
\begin{eqnarray}
\label{fact1}
&&\left\Vert \left(-(A_{h,k}+J^h_k)\right)^{-1-\epsilon}\delta_k^{(1)h}\right\Vert\nonumber\\
&\leq& C\Delta t^3\left\Vert \left(-(A_{h,k}+J^h_k)\right)^{-\epsilon}\chi\left(\Delta t\left(A_{h,k}+J^h_k\right)\right)\left(-(A_{h,k}+J^h_k)\right)^{\epsilon}\right\Vert_{L(H)}\nonumber\\
&&\times\left\Vert\left(-(A_{h,k}+J^h_k)\right)^{-\epsilon}\left(L^h_k\right)'\left(t_k+\frac{\Delta t}{2}\right)\right\Vert\nonumber\\
&\leq& C\Delta t^3t_k^{-1+\beta/2}.
\end{eqnarray}
Note that $J_2$ can be recast in two terms as follows
\begin{eqnarray}
\label{cast1}
J_2&=&-\sum_{k=0}^{m-1}S^h_{m-1,k+1}\varphi_1\left(\Delta t\left(A_{h,k}+J^h_k\right)\right)\delta^{(1)h}_{k+1}\nonumber\\
&-&\sum_{k=0}^{m-1}S^h_{m-1,k+1}\varphi_1\left(\Delta t\left(A_{h,k}+J^h_k\right)\right)\delta^{(2)h}_{k+1}\nonumber\\
&=:& J_{21}+J_{22}.
\end{eqnarray}
Using \lemref{bonlemma}, \eqref{very1},  \lemref{stabilite}, \lemref{lemmastig} and \lemref{lemmestrategie2} it holds that
\begin{eqnarray}
\label{cast2}
&&\Vert J_{21}\Vert\nonumber\\
&\leq& \sum_{k=0}^{m-1}\left\Vert S^h_{m-1,k+1}(-A_h(0))^{1-\epsilon/2}\right\Vert_{L(H)}\nonumber\\
&&\times\left\Vert(-A_h(0))^{-1+\epsilon/2}\varphi_1\left(\Delta t\left(A_{h,k}+J^h_k\right)\right)\left(-(A_{h,k}+J^h_k)\right)^{1+\epsilon/2}\right\Vert_{L(H)}\nonumber\\
&&\times \left\Vert \left(-(A_{h,k}+J^h_k)\right)^{-1-\epsilon/2}\delta^{(1)h}_{k+1}\right\Vert\nonumber\\
&\leq& C\Delta t^{3}\sum_{k=0}^{m-1}t_{m-k-1}^{-1+\epsilon}t_{k+1}^{-1+\beta/2}\left\Vert(-A_h(0))^{-1+\epsilon/2}\left(-(A_{h,k}+J^h_k)\right)^{1-\epsilon/2}\right\Vert_{L(H)} \nonumber\\
&&\times\left\Vert\left(-(A_{h,k}+J^h_k)\right)^{-1+\epsilon/2}\varphi_1\left(\Delta t\left(A_{h,k}+J^h_k\right)\right)\left(-(A_{h,k}+J^h_k)\right)^{1+\epsilon/2}\right\Vert_{L(H)}\nonumber\\
&\leq& C\Delta t^{3-\epsilon}\sum_{k=0}^{m-1}t_{m-k-1}^{-1+\epsilon}t_{k+1}^{-1+\beta/2}\nonumber\\
&\leq& C\Delta t^{2-\epsilon}.\Delta t\sum_{k=0}^{m-1}t_{m-1-k}^{-1+\epsilon}t_{k+1}^{-1+\beta/2}\nonumber\\
&\leq& C\Delta t^{2-\epsilon}t_{m-1}^{-1+\beta+\epsilon}\leq C\Delta t^{2-\epsilon} \Delta t^{-1+\beta/2}\leq C\Delta t^{1+\beta/2-\epsilon}.
\end{eqnarray}
Using \lemref{bonlemma}, \eqref{fact2} and \lemref{stabilite}, it holds that
\begin{eqnarray}
\label{cast3a}
\Vert J_{22}\Vert&\leq& \sum_{k=0}^{m-1}\left\Vert S^h_{m-1,k+1}(-A_h(0))^{1-\epsilon}\right\Vert_{L(H)}\nonumber\\
&&\times\left\Vert(-A_h(0))^{-1+\epsilon}\varphi_1\left(\Delta t\left(A_{h,k}+J^h_k\right)\right)\left(-(A_{h,k}+J^h_k)\right)\right\Vert_{L(H)}\nonumber\\
&&\times \left\Vert \left(-(A_{h,k}+J^h_k)\right)^{-1}\delta^{(2)h}_{k+1}\right\Vert\nonumber\\
&\leq& C\Delta t^3\sum_{k=0}^{m-1} t_{k+1}^{-2+\beta/2}\left\Vert(-A_h(0))^{-1+\epsilon}\left(-(A_{h,k}+J^h_k)\right)^{1-\epsilon}\right\Vert_{L(H)}\nonumber\\
&&\times\left\Vert\left(-(A_{h,k}+J^h_k)\right)^{-1+\epsilon}\varphi_1\left(\Delta t\left(A_{h,k}+J^h_k\right)\right)\left(-(A_{h,k}+J^h_k)\right)\right\Vert_{L(H)}\nonumber\\
&\leq& C\Delta t^3\sum_{k=0}^{m-1}\Delta t^{-\epsilon} t_{k+1}^{-2+\beta/2}\nonumber\\
&\leq&C\Delta t^{2-\epsilon}\,\Delta t\sum_{k=0}^{m-1}t_{k+1}^{-2+\beta/2}.
\end{eqnarray}
Note that 
{\small
\begin{eqnarray}
\label{cast3b}
\Delta t\sum_{k=0}^{m-1}t_{k+1}^{-2+\beta/2}=\Delta t^{-1+\beta/2}\sum_{k=0}^{m-1}(k+1)^{-2+\beta/2}=\Delta t^{-1+\beta/2}\sum_{k=1}^mk^{-2+\beta/2}.
\end{eqnarray}
}
The sequence $v_k=k^{-2+\beta/2}$ is decreasing. Therefore, by comparison with the integral we have
\begin{eqnarray}
\label{cast3c}
\sum_{k=1}^mv_k=\sum_{k=1}^mk^{-2+\beta/2}\leq 1+\int_1^mt^{-2+\beta/2}dt\leq 1+Cm^{-1+\beta/2}.
\end{eqnarray}
Substituting \eqref{cast3c} in \eqref{cast3b} yields
\begin{eqnarray}
\label{cast3d}
\Delta t\sum_{k=0}^{m-1}t_{k+1}^{-2+\beta/2}\leq C\Delta t^{-1+\beta/2}+Ct_m^{-1+\beta/2}.
\end{eqnarray}
Substituting \eqref{cast3d} in \eqref{cast3a} yields
\begin{eqnarray}
\label{cast3}
\Vert J_{22}\Vert\leq C\Delta t^{1+\beta/2-\epsilon}+C\Delta t^{2-\epsilon}t_m^{-1+\beta/2}\leq C\Delta t^{1+\beta/2-\epsilon}.
\end{eqnarray}
Substituting \eqref{cast3} and \eqref{cast2} in \eqref{cast1} yields 
\begin{eqnarray}
\label{cast4}
\Vert J_2\Vert\leq \Vert J_{21}\Vert+\Vert J_{22}\Vert\leq C\Delta t^{1+\beta/2-\epsilon}.
\end{eqnarray}
Substituting \eqref{cast4} in \eqref{cast0} yields
\begin{eqnarray}
\label{cast5}
\Vert e^h_m\Vert\leq C\Delta t^{1+\beta/2-\epsilon}+C\Delta t\sum_{k=0}^{m-1}\Vert e^h_k\Vert.
\end{eqnarray}
Applying the discrete Gronwall's inequality to \eqref{cast5} yields
\begin{eqnarray}
\Vert e^h_m\Vert\leq C\Delta t^{1+\beta/2-\epsilon}.
\end{eqnarray}
This completes the proof of \thmref{mainresult1}.
\section{Numerical simulations}
\label{numericalexperiment}
We consider the following  reactive  advection
diffusion reaction  with diagonal difussion tensor  
\begin{eqnarray}
\label{reactiondif1}
 \dfrac{\partial u}{\partial t}=\left[D(t) \left(\varDelta u-\nabla \cdot(\mathbf{v}u)\right)+\dfrac{e^{-t} u}{\vert u\vert +1}\right],
\end{eqnarray}
with  mixed Neumann-Dirichlet boundary conditions on $\Lambda=[0,L_1]\times[0,L_2]$. 
The Dirichlet boundary condition is $u=1$ at $\Gamma=\{ (x,y) :\; x =0\}$ and 
we use the homogeneous Neumann boundary conditions elsewhere. The initial solution is $u(0)=0$.
To check our theoritical result in \thmref{mainresult1}, we use $D(t)=1+e^{-t}$.
For comparaison with current exponential  Rosenbrock method \cite{Antjd2} for constant operator $A$, we have taken $D(t)=1$.
In \figref{FIGII}, we will use the following notations
\begin{itemize}
 \item 'Magnus-Rosenbrock'  is used for  the errors graph  of the  Magnus  Rosenbrock  scheme 
 for the nonautonomous equation \eqref{reactiondif1}  corresponding to the coefficient $D(t)=1+e^{-t}$.
 \item  'C-Magnus-Rosenbrock' is  used for the errors graph of the novel Magnus  Rosenbrock  scheme for  fixed coefficient $D(t)=1$ in  \eqref{reactiondif1}
 (constant operator linear operator).
 \item 'Exponential-Rosenbrock' is used for the errors graph for the second order exponential Euler Rosenbrock scheme \cite{Antjd2} for fixed coefficient $D(t)=1$ 
 in \eqref{reactiondif1}(constant operator linear operator).
 \end{itemize}
 In all graphs, the reference solution or 'exact solution' is numerical solution with the smaller time step $\Delta t= 1/4096$.
   The linear operator $A(t)$ is given by 
\begin{eqnarray}
A(t)=(1+e^{-t})\left(\varDelta(.)-\nabla. \mathbf{v}(.)\right),\quad t\in[0, T],
\end{eqnarray}
where $\mathbf{v}$ is the Darcy velocity obtained as in \cite[Fig 6]{Antonionews}.
Clearly $\mathcal{D}(A(t))=\mathcal{D}(A(0))$, $t\in[0, T]$ and $\mathcal{D}((-A(t))^{\alpha})=\mathcal{D}((-A(0))^{\alpha})$, $t\in[0, T]$, $0\leq \alpha\leq 1$. 
The  function $q_{ij}(x,t)$ defined in \eqref{family} is given by $q_{ii}(x,t)=1+e^{-t}$, and $q_{ij}(x,t)=0,\, i\neq j$ . Since $q_{ii}(x,t)$ is bounded below by $1+e^{-T}$, 
it follows that the ellipticity  condition \eqref{ellip} holds and therefore as a consequence of \secref{numscheme}, 
it follows that $A(t)$ is sectorial. Obviously \assref{assumption2} is fulfills. 
  The nonlinear function $F$ is given by $F(t,v)= \dfrac{e^{-t} v}{1+\vert v\vert}$, $t\in[0, T]$, $v\in H$ and obviously satisfies \assref{assumption3}.  Let $f:[0, T]\times \Lambda\times\mathbb{R}\longrightarrow \mathbb{R}$  be  defined by $f(t, x, z)=\frac{e^{-t}z}{1+\vert z\vert}$.  We take $F :[0, T]\times H\longrightarrow H$ to be the Nemytskii operator defined as follows
 \begin{eqnarray}
 (F(t, v))(x)=f(t, x, v(x)),\quad t\in[0, T],\quad x\in \Lambda,\quad v\in H.
 \end{eqnarray}
 One can easily check that
 \begin{eqnarray}
 \frac{\partial f}{\partial z}(t, x, z)=-\frac{e^{-t}\vert z\vert}{(1+\vert z\vert)^2},\quad (t, x, z)\in[0, T]\times\Lambda\times\mathbb{R}.
 \end{eqnarray}
 Therefore
 \begin{eqnarray}
 (F'(t, v))(u(x))=\frac{\partial f}{\partial z}(t, x, v(x)).u(x)=-\frac{e^{-t}\vert v(x)\vert}{(1+\vert v(x)\vert)^2}.u(x).
 \end{eqnarray}
One can easily check that 
\begin{eqnarray}
\label{deriveebornee}
\frac{e^{-t}\vert v(x)\vert}{(1+\vert v(x)\vert)^2}\leq \frac{e^{-t}\vert v(x)\vert}{1+\vert v(x)\vert}\leq e^{-t}\leq C,\quad t\in[0, T],\quad x\in \Lambda,\quad v\in H.
\end{eqnarray}
Therefore, it holds that
\begin{eqnarray}
\left\Vert \frac{\partial F}{\partial u}(t,u)\right\Vert_{L(H)}\leq C,\quad -\left\langle F'(t, u)v, v\right\rangle_H\geq 0,\quad t\in[0, T],\quad u,v \in H.
\end{eqnarray}
One can also obviously prove that
\begin{eqnarray}
\left\Vert\frac{\partial^kF }{\partial t\partial u}(t,u)\right\Vert_{L(H)}\leq C,\quad\left\Vert\frac{\partial^2F}{\partial u^2}(t,u)\right\Vert_{L(H\times H; H)}\leq C,\nonumber
\end{eqnarray}
for all $t\in[0, T]$ and $u\in H$. Hence \assref{assumption3} is fulfilled. 

\begin{figure}[!ht]
 \begin{center}
 \includegraphics[width=0.60\textwidth]{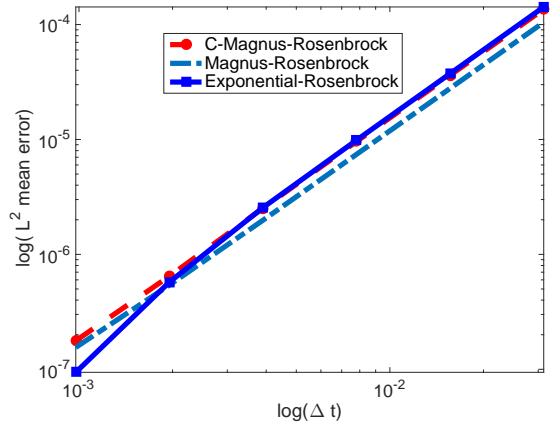}
  \end{center}
 \caption{Convergence of the Magnus  Rosenbrock  scheme  at final time  $T=1$.  For constant coefficient $D(t)=1$, we have
 compared the Magnus  Rosenbrock scheme with  the second order exponential Euler Rosenbrock scheme \cite{Antjd2}.
 The order of convergence in time  is $1. 92$ Magnus  Rosenbrock  scheme (with  $D(t)=1+e^{-t}$), $1.95$ for  the 
 Magnus  Rosenbrock  scheme (with  $D(t)=1$)and  $2.08$  for the second order exponential Euler Rosenbrock scheme.
 }
 \label{FIGII}
 \end{figure}
 In \figref{FIGII}, we can observe the convergence of the Magnus  Rosenbrock scheme ($D(t)=1+e^{-t}$ and   $D(t)=1$), 
 and the second order exponential Euler Rosenbrock scheme ($D(t)=1$). The order of convergence in time  is $1. 92$  for Magnus Rosenbrock  
 scheme ($D(t)=1+e^{-t}$), $1.95$ for  the  Magnus  Rosenbrock  scheme ($D(t)=1$) and $2.08$  for the second order exponential Euler Rosenbrock scheme ($D(t)=1$).
As we can also observe, the convergence orders in time  of  the Magnus  Rosenbrock  scheme are well in agreement  
with our theoretical result in \thmref{mainresult1} as the theoretical order  is $2$
 with order reduction $\epsilon$, which is very small here.
 
\end{document}